\documentclass[aos,MSNbibl,secthm,dvips]{arximspdf}

\doi{10.1214/14-AOS1272} 
\volume{43}
\issue{1}
\pubyear{2015}
\firstpage{177}
\lastpage{214}
\docsubty{FLA}

\makeatletter
\newcommand{\ee}{\mathbb{E}}
\newcommand{\eqref}[1]{(\ref{#1})}
\newcommand{\ep}{\varepsilon}
\newtheorem{lmm}[thm]{Lemma}
\newcommand{\mf}{\mathcal{F}}
\newcommand{\cp}{\mathcal{P}}
\newcommand{\pp}{\mathbb{P}}
\newcommand{\ra}{\rightarrow}
\newcommand{\rr}{\mathbb{R}}
\newcommand{\tr}{\operatorname{Tr}}
\newcommand{\var}{\operatorname{Var}}
\newcommand{\hp}{\hat{p}}
\newcommand{\rank}{\operatorname{rank}}
\makeatother

\begin{document}
\begin{frontmatter}

\title{Matrix estimation by Universal Singular Value~Thresholding}
\runtitle{Matrix estimation by USVT}

\begin{aug}
\author[A]{\fnms{Sourav}~\snm{Chatterjee}\corref{}\ead[label=e1]{souravc@stanford.edu}\thanksref{T1}}
\runauthor{S. Chatterjee}
\affiliation{Stanford University}
\address[A]{Department of Statistics\\
Stanford University\\
Stanford, California 94305\\
USA\\
\printead{e1}}
\end{aug}
\thankstext{T1}{Supported in part by NSF Grant DMS-10-05312.}

\received{\smonth{9} \syear{2013}}
\revised{\smonth{9} \syear{2014}}

\begin{abstract}
Consider the problem of estimating the entries of a large matrix, when
the observed entries are noisy versions of a small random fraction of
the original entries. This problem has received widespread attention in
recent times, especially after the pioneering works of Emmanuel Cand\`es and collaborators. This paper introduces a simple estimation
procedure, called Universal Singular Value Thresholding (USVT), that
works for any matrix that has ``a little bit of structure.''
Surprisingly, this simple estimator achieves the minimax error rate up
to a constant factor. The method is applied to solve problems related
to low rank matrix estimation, blockmodels, distance matrix completion,
latent space models, positive definite matrix completion, graphon
estimation and generalized Bradley--Terry models for pairwise comparison.
\end{abstract}

\end{frontmatter}

\section{Introduction}\label{intro}

Consider a statistical estimation problem where the unknown parameter
is not a single value or vector, but an $m\times n$ matrix $M$. Given
an estimator $\hat{M}$, one choice for a measure of the error in
estimation is the mean-squared error, defined as
%
\begin{equation}
\label{msedef}
\operatorname{MSE}(\hat{M}) := \ee \Biggl[\frac{1}{mn}\sum
_{i=1}^m \sum_{j=1}^n
(\hat{m}_{ij} - m_{ij})^2 \Biggr].
\end{equation}
Here, $\hat{m}_{ij}$ and $m_{ij}$ denote the $(i,j)$th elements of $\hat
{M}$ and $M$, respectively. If we have a sequence of such problems, and
$M_n$ and $\hat{M}_n$ denote the parameter and the estimator in the
$n$th problem, then by usual statistical terminology we may say that
the sequence of estimators $\hat{M}_n$ is consistent if
\[
\lim_{n\ra\infty} \operatorname{MSE}(\hat{M}_n) = 0.
\]
The problem of estimating the entries of a large matrix from incomplete
and/or noisy entries has received widespread attention ever since the
proliferation of large data sets. Early work using spectral analysis
was done by a number of authors in the engineering literature, for
example, by Azar et~al.~\cite{azaretal01} and Achlioptas and \mbox{McSherry}
\cite{achlioptas01}. This was followed by a sizable body of work on
spectral methods, the main pointers to which may be found in the
important recent papers of Keshavan, Montanari and Oh~\cite{kmo10a,kmo10b}. Nonspectral methods also appeared, for example, in~\cite{renniesrebro05}.

In a different direction, statisticians have worked on matrix
completion problems under a variety of modeling assumptions. Possibly
the earliest works are due to Fazel \cite{fazel02} and Rudelson and
Vershynin \cite{rv07}. The emergence of compressed sensing \cite{donoho06,candesrombergtao06} has led to an explosion in activity in
the field of matrix estimation and completion, beginning with the work
of Cand\`es and Recht \cite{candesrecht09}. The pioneering works of
Emmanuel Cand\`es and his collaborators \cite{candesrecht09,candestao10,candesplan10,ccs} introduced the technique of matrix
completion by minimizing the nuclear norm under convex constraints,
which is a convex optimization problem tractable by standard
algorithms. This method has the advantage of \textit{exactly}, rather than
approximately, recovering the entries of the matrix when a suitable low
rank assumption is satisfied, together with a certain other assumption
called ``incoherence.''

Since the publication of \cite{candesrecht09}, a number of statistics
papers have attacked the matrix completion problem from various angles.
Some notable examples are \cite{negahban,mht,klt,rohdetsybakov11,kol2012,davenport}. In a different direction, a paper that seems to
have a close bearing on the analytical aspects of this paper is a
manuscript of Oliveira \cite{oliveira09}.


In addition to the theoretical advances, a large number of algorithms
for matrix completion and estimation have emerged. The main ones are
nicely summarized and compared in \cite{mht}.


The purpose of this paper is to introduce a new estimator that is
capable of solving a variety of matrix estimation problems that are not
tractable by existing tools (at least in a mathematically provable
sense). The estimator and its properties are described in this
introductory section. Section~\ref{examples} focuses on applications,
which include applications to low rank matrices, stochastic
blockmodels, distance matrices, latent space models, positive definite
matrices, graphons and generalized Bradley--Terry models. All proofs
are in Section~\ref{proofs}. An expanded version (version 5) of the
paper containing more theorems, examples and simulation results is
available on arXiv at the URL:
\url{http://arxiv.org/pdf/1212.1247v5.pdf}.

For interesting new developments that appeared after the first draft of
this paper was posted on arXiv, see \cite{choi,dg,nadakuditi}.
Further references and citations are given in subsequent sections.


\subsection{The setup}
Suppose that we have a $m\times n$ matrix $M$, where $m\le n$ and the
entries of $M$ are bounded by $1$ in absolute value. Let $X$ be a
matrix whose elements are independent random variables, and $\ee
(x_{ij})=m_{ij}$ for all $i$ and $j$ [where, as usual, $x_{ij}$ and
$m_{ij}$ denote the $(i,j)$th entries of $X$ and $M$, resp.].
Assume that the entries of $X$ are also bounded by $1$ in absolute
value, with probability one. A matrix such as $X$ will henceforth be
called a ``data matrix with mean $M$.'' The matrix $M$ will sometimes be
called the ``parameter matrix.'' Let $p$ be a real number belonging to
the interval $[0,1]$. Suppose that each entry of $X$ is observed with
probability $p$, and unobserved with probability $1-p$, independently
of the other entries.

The above model will henceforth be referred to as the ``asymmetric
model.'' The ``symmetric model'' is defined in a similar manner: Take any
$n$ and let $M$ be a symmetric matrix of order $n$, whose entries are
bounded by $1$ in absolute value. Let $X$ be a symmetric random matrix
of order $n$ whose elements on and above the diagonal are independent,
and $\ee(x_{ij})=m_{ij}$ for all $1\le i\le j\le n$. As before, assume
that the entries of $X$ are almost surely bounded by $1$ in absolute
value. Take any $p\in[0,1]$ and suppose that each entry of $X$ on and
above the diagonal is observed with probability $p$, and unobserved
with probability $1-p$, independently of the other entries.



Similarly, one can define the ``skew-symmetric model,'' where the
difference $X-M$ is skew-symmetric, with independence on and above the
diagonal as in the symmetric model. This model is used for analyzing
the nonparametric Bradley--Terry model in Section~\ref{bradley}.

\subsection{The USVT estimator}
In the above models, we construct an
estimator $\hat{M}$ of $M$ based on the observed entries of $X$ along
the following steps. Tentatively, I~call this the Universal Singular
Value Thresholding (USVT) algorithm.
\begin{longlist}[5.]
\item[1.] For each $i, j$, let $y_{ij} = x_{ij}$ if $x_{ij}$ is observed,
and let $y_{ij}=0$ if $x_{ij}$ is unobserved. Let $Y$ be the matrix
whose $(i,j)$th entry is $y_{ij}$.
\item[2.] Let $Y=\sum_{i=1}^m s_i u_i v_i^T$ be the singular value
decomposition of $Y$. (In the symmetric and skew-symmetric models, $m=n$.)
\item[3.] Let $\hp$ be the proportion of observed values of $X$. In the
symmetric and skew-symmetric models, let $\hp$ be the proportion of
observed values on and above the diagonal.
\item[4.] Choose a small positive number $\eta\in(0,1)$ and let $S$ be the
set of ``thresholded singular values,'' defined as
\[
S := \bigl\{i \dvtx  s_i \ge(2+\eta) \sqrt{n\hp} \bigr\}.
\]
[Note: (a) In simulations, the method described below seemed to work
even if $\eta$ was taken to be exactly equal to zero; but the
mathematical proof that I have requires $\eta$ to be positive. In
practice, one may choose $\eta$ {a priori} to be some arbitrary
small positive number, say, $0.01$; but a data-dependent choice is not
allowed. (b) If it is known that $\var(x_{ij})\le\sigma^2$ for all
$i,j$, where $\sigma$ is a known constant${}\le1$, then the threshold
$(2+\eta)\sqrt{n\hat{p}}$ may be improved to $(2+\eta)\sqrt{n\hat{q}}$,
where $\hat{q} := \hat{p} \sigma^2 + \hat{p}(1-\hat{p})(1-\sigma^2)$.]%
%
\item[5.] Define
\[
W := \frac{1}{\hp} \sum_{i\in S} s_i
u_i v_i^T.
\]
\item[6.] Let $w_{ij}$ denote the $(i,j)$th element of $W$. Define
\[
\hat{m}_{ij} :=
\cases{ w_{ij}, & \quad$\mbox{if }
-1\le w_{ij} \le 1$,
\cr
1, &\quad$\mbox{if } w_{ij} > 1$,
\cr
-1, &\quad$\mbox{if } w_{ij} < -1$.}
\]
\item[7.] Let $\hat{M}$ be the matrix whose $(i,j)$th entry is $\hat{m}_{ij}$.
\item[8.] If the entries of $M$ and $X$ are known to belong to an interval
$[a,b]$ instead of $[-1,1]$, then subtract $(a+b)/2$ from each entry of
$X$ and divide by $(b-a)/2$, so that the entries are forced to lie in
$[-1,1]$, then apply the above procedure, and finally multiply the
end-result by $(b-a)/2$ and add $(a+b)/2$ to get the estimate of $M$.
\item[9.] If $m > n$, then one should work with $M^T$ and $X^T$ instead of
$M$ and $X$, so that the number of rows is forced to be $\le$ the
number of columns. 
\end{longlist}
%

\subsection{Main result}
Recall that the \textit{nuclear norm} of $M$, written $\|M\|_*$, is
defined as the sum of the singular values of $M$. Recall also the
definition \eqref{msedef} of the mean squared error of a matrix
estimator. The following theorem gives an error bound for the estimator
$\hat{M}$ in terms of the nuclear norm of $M$. This is the main result
of this paper. 

\begin{thm}\label{mainest}
Let $\hat{M}$ and $M$ be as above. Let $\operatorname{MSE}(\hat{M})$ be
defined as in~\eqref{msedef}. Suppose that $p \ge n^{-1+\ep}$ for some
$\ep> 0$. Then
\[
\operatorname{MSE}(\hat{M}) \le C\min \biggl\{\frac{\|M\|_*}{m\sqrt{np}}
, \frac{\|M\|_*^2}{mn} , 1 \biggr\} + C(\ep)e^{-cnp},
\]
where $C$ and $c$ are positive constants that depend only on the choice
of $\eta$ and $C(\ep)$ depends only on~$\ep$ and $\eta$. The same
result holds for the symmetric and skew-symmetric models, after putting $m=n$.

Moreover, if in the same setting as above, we know that $\var
(x_{ij})\le\sigma^2$ for all $i,j$ for some known $\sigma^2\le1$, and
the threshold is set at $(2+\eta)\sqrt{n\hat{q}}$ (see step
$4$ of the algorithm), the same result holds under the condition that
$q \ge n^{-1+\ep}$, where $q:= p\sigma^2 + p(1-p)(1-\sigma^2)$. In this
case the exponential term in the error changes to $C(\ep)e^{-cnq}$ and
the term $\|M\|_*/(m\sqrt{np})$ improves to $\|M\|_* \sqrt{q}/(m\sqrt
{n} p)$. 
\end{thm}

%
Incidentally, the proof shows that the condition $p > n^{-1+\ep}$ may
be improved to $p > n^{-1}(\log n)^{6+\ep}$ (see Theorem~\ref
{normthm}), but I prefer to retain the present version for aesthetic
reasons, especially considering that it is not a real improvement from
any practical point of view.

It should be emphasized that although singular value thresholding has
been used in a number of papers on matrix completion and estimation
(see, e.g., \mbox{\cite{azaretal01,achlioptas01,ccs,kmo10a,kmo10b}} and
references therein), the above algorithm has the unique feature that
the threshold is universal. In the literature, it is usually assumed
that the matrix $M$ has a rank $r$ that is known, and uses the value of
$r$ while thresholding. The USVT algorithm manages to cut off the
singular values at the ``correct'' level, depending on the structure of
the unknown parameter matrix. The adaptiveness of the USVT threshold is
somewhat similar in spirit to that of the \textit{SureShrink} algorithm of
Donoho and Johnstone~\cite{dj95}. SureShrink performs function
estimation by estimating Fourier coefficients in some suitable basis,
and then thresholds the coefficients at a threshold that automatically
adapts to the smoothness of the unknown function. Analogously, the USVT
algorithm computes the eigenvalues of the observed matrix, and then
thresholds the eigenvalues at a universal threshold that is
automatically adaptive in nature, because it picks out as much
``structure'' as is available and throws out all the randomness. This
point will become more clear from the examples discussed in Section~\ref{examples}.


One limitation of USVT is the requirement that the entries should lie
in a bounded interval. One may relax this requirement by assuming, for
example, that the errors $x_{ij}-m_{ij}$ are distributed as normal
random variables with mean zero and variance $\sigma^2$. If $\sigma^2$
is known, then I believe that one can modify the USVT algorithm by
thresholding at $(2+\eta)\sigma\sqrt{n}$ and obtain the same theorems.
The rationale behind this belief is as follows: if $A$ is a large
symmetric random matrix whose entries on and above the diagonal are
independent, have zero mean, and are bounded by $1$ in absolute value,
then the spectral norm of $A$ is less than $2+\eta$ with high
probability. This is the key ingredient in the proof of Theorem~\ref{mainest}.
But such a result continues to be true, after replacing
$2+\eta$ with $(2+\eta)\sigma$, if the entries are normally distributed
with mean zero and variance bounded by $\sigma^2$. Therefore, it is
conceivable that the proof of Theorem~\ref{mainest} may be modified to
accommodate this altered situation. However, if $\sigma^2$ is unknown,
I do not know how to proceed. In reality, $\sigma^2$ will not be known;
this is why I have not worked with the normality assumption. Also, the
situation of normally distributed entries but with a large proportion
missing, seems to be trickier.



\subsection{Minimax lower bound}

It is not difficult to prove that for an $m\times n$ matrix~$M$ with
entries bounded by $1$ in absolute value, where $m\le n$, the nuclear
norm is bounded by $m\sqrt{n}$. Given a number $\delta\in[0,m\sqrt
{n}]$, one may take an arbitrary estimator $\tilde{M}$ and try to find
the $M$ among all $M$ satisfying $\|M\|_*\le\delta$ for which
$\operatorname{MSE}(\tilde{M})$ is maximum. Recall that an estimator that minimizes
this maximum error is classically known as a minimax estimator. The
following theorem shows that our estimator $\hat{M}$ is minimax up to a
constant multiplicative factor and an exponentially small additive discrepancy.

\begin{thm}\label{minimaxthm}
Consider the general matrix estimation problem outlined in the
beginning of this section. Given any estimator $\tilde{M}$ and any
$\delta\in[0, m\sqrt{n}]$, there exists $M$ satisfying $\|M\|_*\le
\delta$ and a\vspace*{1pt} data matrix $X$ with mean $M$, such that for this $M$ and
$X$, the estimator $\tilde{M}$ satisfies
\[
\operatorname{MSE}(\tilde{M}) \ge c \min \biggl\{\frac{\delta}{m\sqrt{np}}, \frac{\delta^2}{mn},
1 \biggr\},
\]
where $c$ is a positive universal constant. Moreover, if $p< 1/2$ then
$X$ and $M$ may be chosen such that $X=M$. The same lower bound holds
in the symmetric case and in the skew-symmetric case. 
\end{thm}

It is worth noting that the exponentially small discrepancy is
necessary. For example, if $\delta=0$, then the minimax error is
obviously zero. However, there is still an exponentially small chance
that $\hat{M}$ may be nonzero. It is also worth noting that if $\delta
$ is not too small (e.g., if $\delta> \sqrt{m/p}$), then the
exponential discrepancy does not matter, and the combination of
Theorems~\ref{mainest} and~\ref{minimaxthm} gives the correct minimax
error up to a universal multiplicative constant.

An examination of the proof of Theorem~\ref{mainest} indicates that
with slight modifications, one may obtain bounds on tail probabilities
instead of an upper bound on the mean squared error. I have retained
the present version for aesthetic reasons.

Incidentally, two notable recent papers, namely, Koltchinskii et
al.~\cite{klt} and Davenport et al.~\cite{davenport}, have suggested
matrix estimation by nuclear norm penalization and proved minimax
optimality results that match up to logarithmic factors. Davenport et
al.~\cite{davenport}, Theorem~3,  show (in the notation of our Theorem~\ref{minimaxthm}) that if the entries of $X$ belong to $\{-1,1\}$ and
if $\delta\ge4\sqrt{mn}$, then the minimax error is bounded below by
a universal constant times $\min\{\delta/(m\sqrt{np}), 1\}$,
provided that this quantity is bigger than $\delta^2/(m^2n)$. This is
almost the same as the conclusion of Theorem~\ref{minimaxthm}, except
that it does not cover the case of $\delta$ smaller than $4\sqrt{mn}$.
Section~3.1 of~\cite{davenport} gives a matrix estimation algorithm
based on nuclear norm penalization that achieves this minimax rate up
to a logarithmic factor. However, the implementation of this algorithm
requires that the user has a reasonable estimate for the nuclear norm
of the unknown matrix $M$, since that is used as the regularization
parameter. USVT has no such requirement. Another advantage that USVT
has over the algorithm of~\cite{davenport} is that it may be easier to
implement, especially for very large matrices, because it does not
involve convex optimization.

The estimator of Koltchinskii et al.~\cite{klt} is also based on
nuclear norm\vspace*{1pt} penalization: translating to our notation, they estimate
$M$ by minimizing $\|X-\hat{M}\|_F^2 + \lambda\|\hat{M}\|_*$ over all
$\hat{M}$, where $\|\cdot\|_F$ is Frobenius norm, $\|\cdot\|_*$ is
nuclear norm, and $\lambda$ is a regularization parameter. It is shown
in \cite{klt} that this problem is actually equivalent to soft singular
value thresholding, where the threshold depends on the parameter
$\lambda$. A conservative choice of $\lambda$ (albeit with an
unspecified constant) and a minimax lower bound that matches the upper
bound up to a logarithmic factor are given in \cite{klt}. The minimax
bound is computed over the set of all matrices with rank less than a
given number and, therefore, is not directly comparable to the minimax
bound in Theorem~\ref{minimaxthm}. With a suitable choice of $\lambda$---but again with unspecified constants---the upper bound in
\cite{klt}, Theorem~3,  becomes (up to a logarithmic factor) essentially equal
to $\|M\|_*/(m\sqrt{np})$. Note that this is the same as the main term
in Theorem~\ref{mainest}. However, if we additionally know that $M$ has
low rank, then the upper bound in \cite{klt}, Theorem~3,  becomes
substantially better (see Section~\ref{lowrank}).


\subsection{Practical issues and warnings}

I do not consider the USVT algorithm as presented above to be in a form
that may implemented ``as is.'' This is mainly for the following reasons:
\begin{longlist}[(a)]
\item[(a)] USVT is minimax optimal only up to a constant factor. In fact, it
is very likely that one may be able to build a better estimator by
taking into account the ratio $m/n$, and getting improved bounds when
this ratio is small. Although Theorem~\ref{minimaxthm} shows that the
improvement will be limited to multiplication by a constant factor,
such an improvement may be important for practical purposes. The recent
paper \cite{dg} has explored the issue of attaining the minimax error
all the way up to the correct constant.

\item[(b)] The number $\eta$ is a ``tuning parameter'' for this algorithm,
that may be chosen by the implementer. The theorem is valid with any
choice of $\eta$ in the interval $(0,1)$, although the constants in the
error bounds blow up as $\eta$ tends to zero. I have noticed in
simulations that taking $\eta=0$ works quite well, but I do not know
how to prove that. Choosing $\eta$ to be a small but fixed positive
number such as $0.01$ is consistent with the requirements of Theorem~\ref{mainest} and seemed to give good results in simulations. Choosing
$\eta$ in a data-dependent manner is, however, not covered by Theorem~\ref{mainest}. 

\item[(c)] Note that in practice, any data matrix may be centered and scaled
so that the entries are forced to lie in the interval $[-1,1]$.
However, if the centering and scaling are done in a data-dependent
manner, then the assertion of Theorem~\ref{mainest} is no longer
guaranteed to be true. 
\end{longlist}

\subsection{An impossibility theorem for error estimates}\label{imposssec}

Theorem~\ref{mainest} gives an upper bound on the mean squared error of
$\hat{M}$. The estimate involves the nuclear norm of parameter matrix
$M$. A natural question is: Is it possible to estimate the true MSE of
$\hat{M}$ from the data?

A straightforward approach is to use parametric bootstrap. Having
estimated $M$ using $\hat{M}$, one may choose a large number $K$,
generate $K$ copies of the data using $\hat{M}$ as the parameter
matrix, compute the estimates $\hat{M}^{(i)}$, $i=1,\ldots, K$ for the
$K$ simulations, and estimate the MSE of $\hat{M}$ using the bootstrap estimator
\[
\widehat{\operatorname{MSE}}_{\mathrm{BS}}(\hat{M}) = \frac{1}{K}\sum
_{i=1}^K \frac{\|\hat{M}^{(i)} - \hat{M}\|_F^2}{mn}.
\]
For the validity of the bootstrap estimate of the MSE, it is essential
that the original $\hat{M}$ is an accurate estimate of $M$. In other
words, we need to know a priori that $\operatorname{MSE}(\hat{M})$ is small to be able
to claim that the bootstrap estimator of $\operatorname{MSE}(\hat{M})$ is accurate.
Theorem~\ref{mainest} implies that if we know that $\|M\|_*$ is small
enough from assumptions, this is true. 

But is it possible to somehow determine whether $\operatorname{MSE}(\hat{M})$ is small
or not from the data, if we do not make any assumption about $M$ to
start with? We will now show that it is impossible to do so, not only
for the estimator $\hat{M}$ but for any ``nontrivial'' estimator $\tilde{M}$.

The definition of a nontrivial estimator is as follows. Given a
parameter matrix $M$ and a data matrix $X$ satisfying the conditions of
Section~\ref{intro}, the trivial estimator of $M$ based on $X$ is
simply $X$ itself. We will denote the trivial estimator as $\hat
{M}^{\mathrm{Trv}}$. Now suppose we are given some estimator $\tilde{M}$.
We will say that the estimator $\tilde{M}$ is nontrivial if there
exists a sequence of parameter matrices $M_n$ and data matrices $X_n$
such that
\[
\operatorname{MSE}\bigl(\hat{M}_n^{\mathrm{Trv}}\bigr) \not\to 0 \qquad  \mbox{as } n\ra \infty,
\]
but $\lim_{n\ra\infty} \operatorname{MSE}(\tilde{M}_n)=0$. In other words,
$\tilde{M}_n$ solves a nontrivial estimation problem. The USVT
estimator is clearly nontrivial, as demonstrated by the examples from
Section~\ref{examples}.

Suppose that we have a nontrivial estimator $\tilde{M}$ and a
procedure $\textup{P}$ that gives an estimate $\widehat{\mathrm
{MSE}}_{\mathrm{P}}(\tilde{M})$ of the MSE of $\tilde{M}$. The MSE
estimate is computed using only the data. The procedure will be called
``good'' if the following two conditions hold:
\begin{longlist}[(a)]
\item[(a)] Whenever $M_n$ is a sequence of parameter matrices and $X_n$
is a sequence of data matrices such that $\operatorname{MSE}(\tilde{M}_n)$ tends to
zero, the estimate $\widehat{\operatorname{MSE}}_{\mathrm{P}}(\tilde{M}_n)$
also tends to zero in probability.
\item[(b)] Whenever $M_n$ is a sequence of parameter matrices and $X_n$
is a sequence of data matrices such that $\operatorname{MSE}(\tilde{M}_n)$ does not
tend to zero,\break $\widehat{\operatorname{MSE}}_{\mathrm{P}}(\tilde{M}_n)$ also
does not tend to zero in probability.
\end{longlist}

In the above setting, the following theorem establishes the
impossibility of the existence of a good estimator for the MSE.

\begin{thm}\label{impossible}
There cannot exist a good procedure for estimating the mean squared
error of a nontrivial estimator.
\end{thm}

\section{Applications}\label{examples}

Throughout this section, $m$, $n$, $M$, $X$, $p$ and $\hat{M}$ will be
as in Section~\ref{intro}. Just to remind the reader, $M$ is an
$m\times n$ matrix where $1\le m\le n$. The entries of $M$ are assumed
to be bounded by $1$ in absolute value. The matrix $X$ is a random
matrix whose entries are independent, and the $(i,j)$th element
$x_{ij}$ has expected value equal to $m_{ij}$, the $(i,j)$th entry of
$M$. Moreover, they satisfy $|x_{ij}|\le1$ with probability one. In
particular, $X$ may be exactly equal to $M$, with no randomness. Each
entry of $X$ is observed with probability $p$ and unobserved with
probability $1-p$, independently of other entries. Occasionally, we
will assume the symmetric model, where $m=n$, and the matrices $M$ and
$X$ are symmetric. In the special case of the Bradley--Terry model in
Section~\ref{bradley}, we will assume the skew-symmetric model, where
$X-M$ is skew-symmetric.

We will now work out various specific cases where Theorem~\ref{mainest}
gives useful results. 

\subsection{Low rank matrices}\label{lowrank}

Estimating low rank matrices has been the focus of the vast majority of
prior work \cite{azaretal01,achlioptas01,fazel02,renniesrebro05,rv07,candesrecht09,candesplan10,candestao10,ccs,negahban,kmo10a,kmo10b,klt,kol2012,mht}.
Theorem~\ref{mainest} works for low rank
matrices. The following theorem, which is a simple corollary of
Theorem~\ref{mainest}, shows that $\hat{M}$ is a good estimate whenever
the rank of $M$ is small compared to~$mp$ (after assuming, as in
Theorem~\ref{mainest}, that $p\ge n^{-1+\ep}$). 

\begin{thm}\label{lowrankthm}
Suppose that $M$ has rank $r$. Suppose that $p\ge n^{-1+\ep}$ for some
$\ep> 0$. Then
\[
\operatorname{MSE}(\hat{M})\le C\min \biggl\{\sqrt{\frac{r}{mp}} , 1 \biggr\} +
C(\ep)e^{-cnp},
\]
where $C$ and $c$ depend only on $\eta$ and $C(\ep)$ depends only on
$\ep$ and $\eta$. Moreover, the same result holds when $M$ and $X$ are
symmetric.
\end{thm}

The term $1/np$ in the error bound is necessary to take care of the
case $r=0$. Even if $M$ is identically zero, the estimator $\hat{M}$
will incur some error due to the (possible) randomness in $X$.


Let us now inspect how the condition $r\ll mp$ compares with available
bounds. In a notable sequence of papers, Keshavan, Montanari and Oh
\cite{kmo10a,kmo10b} obtain the same condition but only if $m$ and $n$
are comparable and the rank is known. Theorem~\ref{lowrankthm}, on the
other hand, works even for ``very rectangular'' matrices where $m\ll n$
and the rank is unknown. 

Cand\`es and Tao~\cite{candestao10} obtain the condition $r\ll mp$ with
an extra poly-logarithmic term in the error. Moreover, they too require
that $m$ and $n$ be comparable, and additionally they need the
so-called ``incoherence condition''. However, as noted before, the
incoherence condition allows exact recovery, while our approach only
gives approximate recovery.

The recent important work of Davenport et~al.~\cite{davenport} gives an
estimator with an error bound that is almost the same as that given by
Theorem~\ref{lowrankthm}, but with a complicated optimization algorithm.

Theorem~\ref{lowrankthm}, however, is probably not an optimal result.
It has been shown by Koltchinskii et al.~\cite{klt}, Theorems~3 and~5, that the true minimax error rate for a closely related problem
is actually $r/mp$, up to a logarithmic factor.

The following theorem shows that the condition $r\ll mp$ is necessary
for estimating $M$.

\begin{thm}\label{lowranklow}
Given any estimator $\tilde{M}$, there exists an $m\times n$ matrix $M$
of rank $r$ with entries bounded between $-1$ and $1$, such that when
the data is sampled from $M$,
\[
\operatorname{MSE}(\tilde{M}) \ge C(1-p)^{[m/r]},
\]
where $C$ is a positive universal constant and $[m/r]$ is the integer
part of~$m/r$.
\end{thm}

\subsection{The stochastic blockmodel}\label{blockmodel}
Consider an undirected graph on $n$ vertices. A stochastic blockmodel
assumes that the vertices $1,\ldots, n$ are partitioned into $k$
blocks, and the probability that vertex $i$ is connected to vertex $j$
by an edge depends only on the blocks to which $i$ and $j$ belong. As
usual, edges are independent of each other. Let $M$ be the matrix whose
$(i,j)$th element is the probability of an edge existing between
vertices $i$ and $j$. The matrix $X$ here is the adjacency matrix of
the observed graph. Here, all elements of $X$ are observed, so $p=1$.

This is commonly known as the stochastic blockmodel. It was introduced
by Holland, Laskey and Leinhardt \cite{hll83} as a simple stochastic
model of social networks. It has become one of the most successful and
widely used models for community structure in networks, especially
after the advent of large data sets.

Early analysis of the stochastic blockmodel was carried out by Snijders
and Nowicki \cite{sn97,sn01}, who provided consistent parameter
estimates when there are exactly two blocks. This was extended to a
finite but fixed number of blocks of equal size by Condon and Karp \cite{condonkarp01}. Bickel and Chen \cite{bickelchen09} were the first to
give consistent estimates for finite number of blocks of unequal size.
It was observed by Leskovec et~al.~\cite{leskovecetal08} that in real
data, the number of blocks often seem to grow with the number of nodes.
This situation was rigorously analyzed for the first time in Rohe
et~al.~\cite{rcy}, and was followed up shortly thereafter by \cite{bcl11,cwa,mossel12,chaudhurietal12} with more advanced results.

However, all in all, I am not aware of any estimator for the stochastic
blockmodel that works whenever the number of blocks is small compared
to the number of nodes. The best result till date is in the very recent
manuscript of Rohe et~al.~\cite{roheetal12}, who prove that a penalized
likelihood estimator works whenever $k$ is comparable to $n$ ``up to log
factors.'' The following theorem says that the USVT estimator $\hat{M}$
gives a complete solution to the estimation problem in the stochastic
blockmodel if $k\ll n$, with no further conditions required. (The
method will not work very well for sparse graphs, however; for recent
advances on estimation in sparse graphs, see \cite{chenetal12}.)

\begin{thm}\label{blockthm}
For a stochastic blockmodel with $k$ blocks,
\[
\operatorname{MSE}(\hat{M}) \le C\sqrt{\frac{k}{n}},
\]
where $C$ is a constant that depends only on our choice of $\eta$.
\end{thm}

Note that estimating the stochastic blockmodel is a special case of low
rank matrix estimation with noise. It is not difficult to prove that
the estimation problem is impossible when $k$ is of the same order as
$n$. We will not bother to write down a formal proof.

\subsection{Distance matrices}\label{distance}
Suppose that $K$ is a compact metric space with metric~$d$. Let
$x_1,\ldots, x_n$ be arbitrary points from $K$, and let $M$ be the
$n\times n$ matrix whose $(i,j)$th entry is $d(x_i, x_j)$. Such
matrices are called ``distance matrices''. Since $K$ is a compact metric
space, the diameter of $K$ with respect to the metric~$d$ must be
finite. Scaling $d$ by a constant factor, we may assume without loss of
generality that the diameter is bounded by $1$, so that the entries of
$M$ are bounded by $1$ as required by Theorem~\ref{mainest}.

Completing a distance matrix with missing entries has been a popular
problem in the engineering and social sciences for a long time; see,
for example, \cite{sd74,bj95,alfakihetal99,biswasetal06,singer08,singer10}. It has become particularly relevant in engineering problems
related to sensor networks. It is also an important issue in
multidimensional scaling \cite{borggroenen10}. For some recent
theoretical advances, see~\cite{ohetal10,jm11}.

In general, distance matrices need not be of low rank. Therefore, much
of the literature on matrix estimation and completion does not apply to
distance matrices. Surprisingly, Theorem~\ref{mainest} gives a complete
solution of the distance matrix completion and estimation problem.



\begin{thm}\label{compact1}
Suppose that $p\ge n^{-1+\ep}$ for some $\ep> 0$. If $M$ is a distance
matrix as above, then
\[
\operatorname{MSE}(\hat{M}) \le\frac{C(K, d, n)}{\sqrt{p}} + C(\ep)e^{-cnp},
\]
where $c$ depends only on $\eta$, $C(\ep)$ depends only on $\ep$ and
$\eta$, and $C(K,d,n)$ is a~number depending only on $K$, $d$, $n$ and
$\eta$ such that
\[
\lim_{n\ra\infty} C(K,d,n)=0.
\]
\end{thm}

The above theorem is not wholly satisfactory, since it does not
indicate how fast $p$ can go to zero as $n\ra\infty$ so that $\hat{M}$
is still consistent. To understand that, we need to know more about the
structure of the space $K$. The following theorem gives a quantitative estimate.

\begin{thm}\label{compact2}
Suppose that for each $\delta> 0$, $N(\delta)$ is a number such that
$K$ may be covered by $N(\delta)$ open $d$-balls of radius $\delta$. Then
\[
\operatorname{MSE}(\hat{M}) \le C\inf_{\delta> 0} \min \biggl\{
\frac{\delta
+ \sqrt{N(\delta/4)/n}}{\sqrt{p}} , 1 \biggr\} + C(\ep)e^{-cnp},
\]
where $C$ and $c$ depend only on $\eta$ and $C(\ep)$ depends only on
$\ep$ and $\eta$.
\end{thm}

To see how Theorem~\ref{compact2} may be used, suppose that $K$ is a
compact subset of the real line and $d$ is the usual distance on $\rr$,
scaled by a factor to ensure that the diameter of $K$ is ${\le}1$. Then
$N(\delta)$ increases like $1/\delta$ as $\delta\ra0$. Consequently,
given~$n$, the optimal choice of $\delta$ is of the order $n^{-1/3}$,
which gives the bound
\[
\operatorname{MSE}(\hat{M})\le\frac{Cn^{-1/3}}{\sqrt{p}}.
\]
(Note that the\vspace*{1pt} exponential term need not appear because the main term
is bounded below by a positive constant if $p< n^{-2/3}$.) Thus, $\hat
{M}$ is a consistent estimate as long as $p$ goes to zero slower than
$n^{-2/3}$ as $n\ra\infty$.

\subsection{Latent space models}\label{latent}
Suppose that $\beta_1,\ldots,\beta_n$ are vectors belonging to some
bounded closed set $K\subseteq\rr^k$, where $k$ is some arbitrary but
fixed dimension. Let $f \dvtx K\ra[-1,1]$ be a continuous function. Let $M$
be the $n\times n$ matrix whose $(i,j)$th element is $f(\beta_i, \beta
_j)$. Then our data matrix $X$ has the form
\[
x_{ij}=f(\beta_i, \beta_j)+
\ep_{ij},
\]
where $\ep_{ij}$ are independent errors with zero mean, satisfying the
restriction that $|x_{ij}|\le1$ almost surely. For example, $X$ may be
the adjacency matrix of a random graph where the probability of an edge
existing between vertices $i$ and $j$ is $f(\beta_i, \beta_j)$. This is
one context where latent space models are widely used, starting with
the work of Hoff, Raftery and Handcock \cite{hrh02}. A large body of
work applying the latent space approach to real data has grown in the
last decade. On the theoretical side, it was observed in \cite{bickelchen09,bcl11} that the latent space model arises naturally from
an exchangeability assumption due to the Aldous--Hoover theorem \cite{aldous81,hoover82}. Note that distance matrices and stochastic
blockmodels are both special cases of latent space models.

There have been various attempts to estimate parameters in the latent
space models (e.g., \cite{hrh02,hrt07,airoldietal08}). Almost all of
these approaches rely on heuristic arguments and justification through
simulations. The problem is that in addition to the vectors $\beta
_1,\ldots,\beta_n$, the function $f$ itself is an unknown parameter. If
either $\beta_i$'s are known, or $f$ is known, the estimation problem
is tractable. For example, when $f(x,y)$ is of the form
$e^{x+y}/(1+e^{x+y})$, the problem was solved in \cite{cds}. However,
when both $f$ and $\beta_i$'s are unknown, the problem becomes
seemingly intractable. In particular, there is an identifiability issue
because $f(x,y)$ may be replaced by $h(x,y) := f(g(x), g(y))$ and $\beta
_i$ by $g^{-1}(\beta_i)$ for any invertible function $g$ without
altering the model.

In view of the above discussion, it is a rather surprising consequence
of Theorem~\ref{mainest} that it is possible to estimate the numbers
$f(\beta_i,\beta_j)$, $i,j=1,\ldots, n$ from a single realization of
the data matrix, under no additional assumptions than the stated ones.

\begin{thm}\label{latent1}
Suppose that $p\ge n^{-1+\ep}$. If $M$ is as above, then
\[
\operatorname{MSE}(\hat{M})\le\frac{C(K, k, f,n)}{\sqrt{p}} + C(\ep)e^{-cnp},
\]
where $c$ depends only on $\eta$, $C(\ep)$ depends only on $\ep$ and
$\eta$, and $C(K, k, f, n)$ depends only on $K$, $k$, $f$, $n$ and $\eta
$ such that
\[
\lim_{n\ra\infty} C(K,k,f,n)=0.
\]
\end{thm}

The problem with Theorem~\ref{latent1}, just like Theorem~\ref{compact1} in Section~\ref{distance}, is that it does not give an
explicit error bound, which makes it impossible to determine how fast
$p$ can go to zero with $n$ so that consistency holds. Again, this is
easy to fix by assuming smoothness properties of $f$ and applying
Lemma~\ref{function}. As a particular example, suppose that $f$ is
Lipschitz with Lipschitz constant $L$, in the sense that
\[
\bigl|f(x,y)-f\bigl(x', y'\bigr)\bigr| \le L
\bigl\|x-x'\bigr\|+L \bigl\|y-y'\bigr\|
\]
for all $x,y,x',y'\in K$.

\begin{thm}\label{latent2}
In the above setting,
\[
\operatorname{MSE}(\hat{M}) \le C(K,k,L) \frac{n^{-1/(k+2)}}{\sqrt{p}},
\]
where $C(K,k,L)$ is a constant depending only on $K$, $k$, $L$ and $\eta$.
\end{thm}

\subsection{Positive definite matrices}\label{positive}
Assume that $m=n$ and $M$ is positive semi-definite. (In the
statistical context, this is the same as saying that $M$ is a
covariance matrix. When the diagonal entries are all $1$, $M$ is a
correlation matrix.)

Completing positive definite matrices with missing entries has received
a lot of attention in the linear algebra literature \cite{groneetal84,johnson90,bhatia08}, although most of the techniques are applicable
only for relatively small matrices or when a sizable fraction of the
entries are observed. In the engineering sciences, estimation of
covariance matrices from a small subset of observed entries arises in
the field of remote sensing (see \cite{candesplan10,candesrecht09,candestao10} for brief discussions).

The statistical matrix completion literature cited in Section~\ref{intro} applies only to low rank positive definite matrices. It is
therefore quite a surprise that the completion problem may be solved
for \textit{any} positive definite matrix whenever we get to observe a
large number of entries from each row. 

\begin{thm}\label{pdthm}
Suppose that $m=n$ and $M$ is positive semi-definite. Suppose that
$p\ge n^{-1+\ep}$. Then
\[
\operatorname{MSE}(\hat{M}) \le\frac{C}{\sqrt{np}} + C(\ep) e^{-cnp},
\]
where $C$ and $c$ depend only on $\eta$ and $C(\ep)$ depends only on
$\ep$ and $\eta$.
\end{thm}

What if $p$ is of order $1/n$ or less? The following theorem shows that
it is impossible to estimate $M$ in this situation.

\begin{thm}\label{pdlow}
Given any estimator $\tilde{M}$, there exists a correlation matrix $M$
such that when the data is sampled from $M$,
\[
\operatorname{MSE}(\tilde{M}) \ge C(1-p)^n,
\]
where $C$ is a positive universal constant.
\end{thm}

\subsection{Graphon estimation}\label{graphons}
A graphon is a measurable function $f$ from $[0,1]^2$ into $[0,1]$ that
satisfies $f(x,y)\equiv f(y,x)$. The term ``graphon'' was coined by Lov\'asz and coauthors in the growing literature on limits of dense graphs
\cite{borgsetal06,borgsetal08,borgsetal07,lovaszszegedy06,lovaszbook}. Such functions also arise in the related study of weakly
exchangeable random arrays \cite{diaconisjanson08,austin08,aldous81,hoover82}. They have also appeared recently in large
deviations \cite{cv,cv2,lubetzky12} and mathematical statistics \cite{cd3,radinyin}.

In the graph limits literature, graphons arise as limits of graphs with
increasing number of nodes. Conversely, graphons are often used to
generate random graphs in a natural way. Take any $n$ and let
$U_1,\ldots, U_n$ be i.i.d. $\operatorname{Uniform}[0,1]$ random variables. Construct
a random undirected graph on $n$ vertices by putting an edge between
vertices $i$ and $j$ with probability $f(U_i, U_j)$, doing this
independently for all $1\le i< j\le n$. This procedure is sometimes
called ``sampling from a graphon'' (see \cite{borgsetal08}, Section~4.4).


The statistical question is the following: Suppose that we have a
random graph on $n$ vertices that is sampled from a graphon. Is it
possible to estimate the graphon from a single realization of the
graph? More precisely, is it possible to accurately estimate the
numbers $f(U_i, U_j)$, $1\le i< j\le n$, from a single realization of
the random graph? The question is similar to the one investigated in
Section~\ref{distance}, but the difference is that here we are not
allowed to assume any regularity on $f$ except measurability.

Taking things back to our usual setting, let $M$ be the matrix whose
$(i,j)$th element is $f(U_i,U_j)$. Note that unlike our previous
examples, $M$ is now random. So the definition of MSE should be
modified to take expectation over $M$ as well.

\begin{thm}\label{graphonthm}
In the above setting,
\[
\operatorname{MSE}(\hat{M}) \le C(f,n),
\]
where $C(f,n)$ is a constant depending only on $f$, $n$ and $\eta$,
such that
\[
\lim_{n\ra\infty} C(f,n)=0.
\]
\end{thm}

Incidentally, after the first version of this paper was put up on
arXiv, several papers (e.g.,~\cite{wolfe1,yang}) on graphon
estimation, advocating a number of different techniques and
demonstrating applications in the statistical study of networks, have
appeared in the literature.

\subsection{Nonparametric Bradley--Terry model}\label{bradley}
Suppose there are $n$ teams playing against each other in a tournament.
Every team plays against every other team at least once (often, exactly
once). Suppose that $p_{ij}$ is the probability that team $i$ wins
against team $j$ in a match between $i$ and $j$. Then $p_{ji}=1-p_{ij}$.

The Bradley--Terry model \cite{bt52}, originally proposed by Zermelo
\cite{zermelo29}, assumes that $p_{ij}$ is of the form $a_i/(a_i+a_j)$
for some unknown nonnegative numbers $a_1,\ldots, a_n$. It is known
how to estimate the parameters $a_1,\ldots, a_n$ if we assume that the
outcomes of all games are independent---which, in this case, is a
reasonable assumption.

The Bradley--Terry model has found great success among practitioners.
For an old survey of the literature on the model dating back to 1976,
see~\cite{df76}. Numerous extensions and applications have been
proposed, for example, \cite{ht98,agresti90,raokupper67,plackett75,luce59,luce77,huangetal06}. The monographs of David \cite{david88}
and Diaconis \cite{diaconis88}, Chapter~9,  explain the statistical
foundations of these models. More recently, several authors have
proposed to perform Bayesian inference for (generalized) Bradley--Terry
models \cite{adams05,gormleymurphy08,gormleymurphy09,goruretal06,guiversnelson09,carondoucet12}.

For the basic Bradley--Terry model, it is possible to find the maximum
likelihood estimate of the $a_i$'s using a simple iterative procedure
\cite{zermelo29,hunter04,langeetal00}. The maximum likelihood estimate
was shown to be jointly consistent for all $n$ parameters by Simons and
Yao \cite{simonsyao99}.


We now generalize the Bradley--Terry model as follows. Suppose, as
before, that $p_{ij}$ is the probability that team $i$ beats team $j$.
Suppose that the teams have a particular ordering in terms of strength
that is unknown to the observer. Assume that \textit{if team $i$ is
stronger than team $j$, then $p_{ik} \ge p_{jk}$ for all $k\ne i,j$.}
Do not assume anything else about the $p_{ij}$'s; in particular, do not
assume any formula for the $p_{ij}$'s in terms of hidden parameters.
This is what we may call a ``nonparametric Bradley--Terry model.'' Note
that the usual Bradley--Terry model is a special case of the
nonparametric version.

In the nonparametric Bradley--Terry model, is it possible to estimate
all the $p_{ij}$'s from a tournament where every team plays against
every other exactly once? Is it possible to estimate the $p_{ij}$'s if
only a randomly chosen fraction of the games are played? How small can
this fraction be, so that accurate estimation is still possible? The
following theorem provides some answers.

\begin{thm}\label{bradleythm}
Consider the nonparametric Bradley--Terry model defined above. Let $M$
be the matrix whose $(i,j)$th entry is $p_{ij}$ if $i\ne j$ and $0$ if
$i=j$. Let $X$ be the data matrix whose $(i,j)$th entry is $1$ if team
$i$ won over team $j$, $0$ if team $j$ won over team $i$ and recorded
as missing if team $i$ did not play versus team $j$. If team $i$ has
played against team $j$ multiple times, let the $(i,j)$th entry of $X$
be the proportion of times that $i$ won over $j$. (Draws are not
allowed.) Let all diagonal entries of $X$ be zero. Given $p\in[0,1]$,
suppose that for each $i$ and $j$, the game between $i$ and $j$ takes
place with probability $p$ and does not take place with probability
$1-p$, independent of other games. Let $\hat{M}$ be the estimate of $M$
based on the data matrix $X$. Then
\[
\operatorname{MSE}(\hat{M}) \le\frac{Cn^{-1/4}}{\sqrt{p}}, 
\]
where $C$ depends only on our choice of $\eta$. In particular, the
estimation problem is solvable whenever $p\gg n^{-1/2}$.
\end{thm}

A natural question is whether the threshold $p\gg n^{-1/2}$ is sharp. I
do not know the answer to this question. 

\section{Proofs}\label{proofs}
\subsection{Proof of Theorem~\texorpdfstring{\protect\ref{mainest}}{1.1} (Main result)}

We need to recall some background material before embarking on the
proof of Theorem~\ref{mainest}.

\subsubsection*{Matrix norms}
Let $A = (a_{ij})_{1\le i\le m,  1\le j\le n}$ be an $m\times n$ real
matrix with singular values $\sigma_1,\ldots, \sigma_k$, where $k= \min
\{m,n\}$. The following matrix norms are widely used in this proof.

The nuclear norm or the trace norm of $A$ is defined as
\[
\|A\|_* := \sum_{i=1}^k
\sigma_i.
\]
The Frobenius norm, also called the Hilbert--Schmidt norm, is defined as
\[
\|A\|_F := \Biggl(\sum_{i=1}^m
\sum_{j=1}^n a_{ij}^2
\Biggr)^{1/2} = \bigl(\tr \bigl(A^TA\bigr)
\bigr)^{1/2} = \Biggl(\sum_{i=1}^k
\sigma_i^2 \Biggr)^{1/2}.
\]
By the Cauchy--Schwarz inequality,
%
\begin{equation}
\label{csineq0}
\|A\|_*\le\sqrt{\rank(A)} \|A\|_F.
\end{equation}
%
The sup-norm is defined as
\[
\|A\|_\infty:= \max_{i,j} |a_{ij}|.
\]
The spectral norm or the operator norm of $A$ is defined as
\[
\|A\| := \max_{1\le i\le k} |\sigma_i|.
\]
The spectral norm may be alternatively expressed as
%
\begin{equation}
\label{maxrep}
\|A\| = \max_{x\in\mathbb{S}^{m-1},  y\in\mathbb{S}^{n-1}} x^TAy,
\end{equation}
where $\mathbb{S}^{m-1}$ and $\mathbb{S}^{n-1}$ are the Euclidean unit
spheres in $\rr^m$ and $\rr^n$, respectively. The above representation
implies that the spectral norm satisfies the triangle inequality.
Consequently, for any two $m\times n$ matrices $A$ and $B$,
\[
\bigl|\|A\| - \|B\| \bigr| \le\|A-B\|\le\|A-B\|_F.
\]
In particular, the spectral norm is a Lipschitz function of the matrix
entries (with Lipschitz constant $1$), if the entries are collectively
considered as a vector of length~$mn$.

The triangle inequality for the spectral norm also implies that the map
$A\mapsto\|A\|$ is convex. Indeed, for any $0\le t\le1$,
\[
\bigl\|tA + (1-t)B\bigr\|\le t\|A\|+(1-t)\|B\|.
\]
For more on matrix norms, see \cite{bhatia97}.

\subsubsection*{Perturbation of singular values}
The following perturbative result from matrix analysis is used several
times in this manuscript. Let $A$ and $B$ be two $m\times n$ matrices.
Let $k=\min\{m,n\}$. Let $\sigma_1,\ldots,\sigma_k$ be the singular
values of $A$ in decreasing order and repeated by multiplicities, and
let $\tau_1,\ldots, \tau_k$ be the singular values of $B$ in decreasing
order and repeated by multiplicities. Let $\delta_1,\ldots,\delta_k$ be
the singular values of $A-B$, in any order but still repeated by
multiplicities. 

\begin{thm}\label{wielandt}
For any $1\le p< \infty$,
\[
\sum_{i=1}^k |\sigma_i-
\tau_i|^p \le\sum_{i=1}^k
|\delta_i|^p
\]
and
\[
\max_{1\le i\le k} |\sigma_i - \tau_i|\le
\max_{1\le i\le k} |\delta_i|.
\]
\end{thm}

The above result follows, for example, from a combination of\break Theorem~III.4.4 and Exercise~II.1.15 in \cite{bhatia97}. It may also be derived
as a consequence of Wielandt's minimax principle \cite{bhatia97}, Section III.3, or Lidskii's theorem
\cite{bhatia97}, Exercise III.4.3. The case $p=2$ is sometimes called the
Hoffman--Wielandt theorem \cite{agz}, Lemma~2.1.19 and Remark~2.1.20,
and the inequality involving the maximum is sometimes called Weyl's
perturbation theorem \cite{bhatia97}, Corollary III.2.6.

\subsubsection*{Bernstein's inequality}
The following inequality is known as ``Bernstein's inequality.''

\begin{thm}\label{bern}
Suppose that $X_1,\ldots, X_n$ are independent random variables with
zero mean, and $M$ is a constant such that $|X_i|\le M$ with
probability one for each $i$. Let $S:=\sum_{i=1}^n X_i$ and $v:= \var
(S)$. Then for any $t\ge0$,
\[
\pp\bigl(|S|\ge t\bigr) \le2\exp \biggl(-\frac{3t^2}{6v +2Mt} \biggr).
\]
\end{thm}

This inequality was proved by Bernstein \cite{bernstein}. For a
discussion of Bernstein's inequality and improvements, see Bennett~\cite{bennett62}.

\subsubsection*{Talagrand's concentration inequality}
Recall that a median $m$ of a random variable $Y$ is a real number such
that $\pp(Y\le m) \ge1/2$ and $\pp(Y\ge m)\ge1/2$. The median may not
be unique.

The following concentration inequality is one of the several striking
inequalities that are collectively known as ``Talagrand's concentration
inequalities.''

\begin{thm}\label{tala}
Suppose that $f\dvtx  [-1,1]^n \ra\rr$ is a convex Lipschitz function with
Lipschitz constant $L$. Let $X_1,\ldots,X_n$ be independent random
variables taking value in $[-1,1]$. Let $Y:= f(X_1,\ldots, X_n)$ and
let $m$ be a median of~$Y$. Then for any $t\ge0$,
\[
\pp\bigl(|Y-m|\ge t\bigr) \le4e^{-t^2/16L^2}.
\]
\end{thm}

For a proof of Theorem~\ref{tala}, see \cite{talagrand96}, Theorem~6.6.

It is easy to modify Theorem~\ref{tala} to have concentration around
the mean instead of the median. Just observe that by Theorem~\ref{tala}, $\ee(Y-m)^2 \le64 L^2$.
Since $\ee(Y-m)^2\ge\var(Y)$, this
shows that $\var(Y)\le64L^2$. Thus, by Chebychev's inequality,
\[
\pp\bigl(\bigl|Y-\ee(Y)\bigr|\ge16L\bigr)\le\tfrac{1}{4}.
\]
By the definition of a median, this shows that $\ee(Y)-16L\le m \le\ee
(Y)+16L$. Together with Theorem~\ref{tala}, this implies that for any
$t\ge0$,
%
\begin{equation}
\label{tala2}
\pp\bigl(\bigl|Y-\ee(Y)\bigr|\ge16L + t\bigr) \le 4e^{-t^2/2L^2}.
\end{equation}
The above inequality has a number of uses in the proof of Theorem~\ref{mainest}.

\subsubsection*{Spectral norms of random matrices}
The following bound on spectral norms of random matrices is a crucial
ingredient for this paper. The proof follows from a combinatorial
argument of Vu \cite{vu07} (which is itself a refinement of a classical
argument of F\"uredi and Koml\'os \cite{furedikomlos81}), together with
Talagrand's inequality \eqref{tala2}.

\begin{thm}\label{normthm}
Take any two numbers $m$ and $n$ such that $1\le m\le n$. Suppose that
$A = (a_{ij})_{1\le i\le m,  1\le j\le n}$ is a matrix whose entries
are independent random variables that satisfy, for some $\sigma^2\in[0,1]$,
\[
\ee(a_{ij})=0, \qquad \ee\bigl(a_{ij}^2\bigr)\le
\sigma^2 \quad \mbox{and}\quad  |a_{ij}|\le 1 \qquad a.s.
\]
Suppose that $\sigma^2 \ge n^{-1+\ep}$ for some $\ep> 0$. Then for any
$\eta\in(0,1)$,
\[
\pp\bigl(\|A\|\ge(2+\eta)\sigma\sqrt{n}\bigr)\le C_1(
\ep)e^{-C_2\sigma^2n},
\]
where $C_1(\ep)$ depends only on $\ep$ and $\eta$ and $C_2$ depends
only on $\eta$. The same result is true when $m=n$ and $A$ is symmetric
or skew-symmetric, with independent entries on and above the diagonal,
all other assumptions remaining the same. Lastly, all results remain
true if the assumption $\sigma^2 \ge n^{-1+\ep}$ is changed to $\sigma
^2 \ge n^{-1}(\log n)^{6+\ep}$.
\end{thm}

\begin{pf}
First assume that $m=n$ and $A$ is symmetric. Note that for any even
number $k$,
%
\begin{equation}
\label{vu1}
\ee\|A\|^{k} \le\ee\bigl(\tr\bigl(A^{k}
\bigr)\bigr) = \sum_{1\le i_1,\ldots, i_{k}\le n} \ee(a_{i_1i_2}a_{i_2i_3}
\cdots a_{i_{k-1}i_{k}} a_{i_{k} i_1}).
\end{equation}
Consider $i_1,i_2,\ldots, i_{k-1},i_k, i_1$ as a closed tour of a
subset of the vertices of the complete graph on $n$ vertices (with
self-edges included). From the given assumptions about the $a_{ij}$'s,
it follows that the term $\ee(a_{i_1i_2}a_{i_2i_3}\cdots a_{i_ki_1})$
is zero if there is an edge that is traversed exactly once. Suppose
that each edge in the tour is traversed at least twice. Let $p$ be the
number of distinct vertices visited by the tour. Then the number of
distinct edges traversed by the tour is at least $p-1$. Since $\sigma^2
\le1$, $|a_{ij}|\le1$, and $\ee|a_{ij}|^l\le\sigma^2$ for any $l\ge
2$, this shows that
%
\begin{equation}
\label{vu2}
\bigl|\ee(a_{i_1i_2}a_{i_2i_3}\cdots a_{i_ki_1})\bigr| \le
\sigma^{2p-2}.
\end{equation}
%
Thus, if $W(n,k,p)$ is the number of tours of length $k$ that visit
exactly $p$ vertices and traverse each of its edges at least twice, then
%
\begin{equation}
\label{vu3}
\ee\|A\|^k \le\sum_{p=1}^{k}
\sigma^{2p-2} W(n,k,p).
\end{equation}
Vu \cite{vu07}, equation (5),  proves that if $p> k/2$ then $W(n,k,p)=0$
and if $p\le k/2$ then
\[
W(n,k,p)\le n(n-1)\cdots(n-p+1)\pmatrix{k \cr 2p-2} p^{2(k-2p+2)}2^{2p-2}.
\]
Using this bound, one can proceed as in \cite{vu07}, Section~2,  to
arrive at the conclusion that if $k$ is largest even number $\le\sigma
^{1/3}n^{1/6}$, then
\[
\ee\|A\|^k\le2n(2\sigma\sqrt{n})^k.
\]
Consequently,
\[
\ee\|A\|\le \bigl(\ee\|A\|^k \bigr)^{1/k}
\le(2n)^{1/k} 2\sigma\sqrt{n}.
\]
This shows that if $\sigma^2 \ge n^{-1+\ep}$ [or if $\sigma^2 \ge
n^{-1}(\log n)^{6+\ep}$], then there is a constant $C(\ep)$ depending
only on $\ep$ and $\eta$ such that if $n\ge C(\ep)$ then
%
\begin{equation}
\label{eabd}
\ee\|A\|\le(2+\eta/4)\sigma\sqrt{n}.
\end{equation}
Since $a_{ij}$ are independent and $|a_{ij}|\le1$ almost surely for
all $i,j$, and the spectral norm is a convex Lipschitz function of
matrix entries with Lipschitz constant $1$ (by the discussion about
matrix norms at the beginning of this section), therefore one can apply
Talagrand's inequality [Theorem~\ref{tala} and inequality \eqref{tala2}] together with~\eqref{eabd} and the assumption that $\sigma^2
\ge n^{-1+\ep}$ to conclude that there is a constant $C(\ep)$ such that
if $n \ge C(\ep)$ then
%
\begin{equation}
\label{pac1c2}
\pp\bigl(\|A\|\ge(2+\eta/2)\sigma\sqrt{n}\bigr)\le
C_1e^{-C_2\sigma^2n},
\end{equation}
where $C_1$ and $C_2$ depend only on $\eta$. Replacing $C_1$ by a large
enough constant $C_1(\ep)$, the condition $n\ge C(\ep)$ may be dropped.
It is clear from the argument that it goes through in the
skew-symmetric case as well.

Let us now drop the assumption of symmetry, but retain the assumption
that $m=n$. Let $a_{ij}':= a_{ji}$. Then inequality \eqref{vu1} must be
modified to say that for any even $k$,
\begin{eqnarray*}
\ee\|A\|^k &\le & \ee \bigl(\tr\bigl(\bigl(A^TA
\bigr)^{k/2}\bigr) \bigr)
\\
&=&  \sum_{1\le i_1,\ldots, i_{k}\le n}\ee \bigl(a_{i_1i_2}'a_{i_2i_3}a_{i_3i_4}'a_{i_4i_5}
\cdots a_{i_{k-1}i_{k}}' a_{i_{k} i_1}\bigr).
\end{eqnarray*}
As before, the term inside the sum is zero for any tour that traverses
an edge exactly once. (In fact, there are more terms that are zero now;
a term may be zero even if a tour traverses all of its edges at least
twice.) Similarly, inequalities \eqref{vu2} and \eqref{vu3} continue to
hold and, therefore, so does the rest of the argument.

Lastly, consider the case $m< n$. Augment the matrix $A$ by adding an
extra $n-m$ rows of zeros to make it an $n\times n$ matrix that
satisfies all the conditions of the theorem. Clearly, the new matrix
has the same spectral norm as the old one. This completes the proof.
\end{pf}

\subsubsection*{The key lemma}
Suppose that $A$ and $B$ are two $m\times n$ matrices, where $m\le n$.
Let $a_{ij}$ be the $(i,j)$th entry of $A$ and $b_{ij}$ be the
$(i,j)$th entry of $B$. It is easy to see from definition that
\[
\frac{1}{mn}\sum_{i=1}^m \sum
_{j=1}^n (a_{ij}-b_{ij})^2
= \frac{1}{mn}\| A-B\|^2_F \le\frac{1}{n}
\|A-B\|^2.
\]
Thus, if $\|A-B\|$ is small enough, then the entries of $A$ are
approximately equal to the entries of $B$, on average. In other words,
the matrix $A$ is an estimate of the matrix $B$.

The goal of this section is to show that if in addition to the
smallness of $\|A-B\|$, we also know that the nuclear norm $\|B\|_*$ is
not too large, it is possible to get a better estimate of $B$ based on $A$.

\begin{lmm}\label{estlmm}
Let $A=\sum_{i=1}^m \sigma_i x_i y_i^T$ be the singular value
decomposition of~$A$. Fix any $\delta> 0$ and define
\[
\hat{B} := \sum_{i  \dvtx   \sigma_i > (1+\delta)\|A-B\|} \sigma_i
x_i y_i^T.
\]
Then
\[
\|\hat{B}-B\|_F\le K(\delta) \bigl(\|A-B\|\|B\|_* \bigr)^{1/2},
\]
where $K(\delta)= (4+2\delta)\sqrt{2/\delta} + \sqrt{2+\delta}$.
\end{lmm}

\begin{pf}
Let $B = \sum_{i=1}^m \tau_i u_i v_i^T$ be the singular value
decomposition of $B$. Without loss of generality, assume that $\sigma
_i$'s and $\tau_i$'s are arranged in decreasing order. Let $S$ be the
set of $i$ such that $\sigma_i > (1+\delta)\|A-B\|$. Define
\[
G := \sum_{i\in S} \tau_i u_i
v_i^T.
\]
Note that by the definition of $\hat{B}$, the largest singular value of
$A-\hat{B}$ is bounded above by $(1+\delta)\|A-B\|$. In other words,
%
\begin{equation}
\label{ab1}
\|A-\hat{B}\|\le(1+\delta)\|A-B\|.
\end{equation}
On the other hand, by Theorem~\ref{wielandt},
\[
\max_{1\le i\le m}|\sigma_i -\tau_i|\le\|A-B
\|.
\]
In particular, for $i\notin S$,
%
\begin{equation}
\label{tauineq}
\tau_i \le\sigma_i + \|A-B\| \le(2+
\delta)\|A-B\|,
\end{equation}
and for $i\in S$,
%
\begin{equation}
\label{tau2}
\tau_i \ge\sigma_i - \|A-B\|\ge\delta
\|A-B\|.
\end{equation}
By \eqref{tauineq},
%
\begin{equation}
\label{ab2}
\|B-G\|\le(2+\delta)\|A-B\|.
\end{equation}
By \eqref{ab1} and \eqref{ab2}, we have
%
\begin{equation}
\label{ab3}
\|\hat{B}-G\|  \le\|\hat{B}-A\|+ \|A-B\|+\|B-G\|\le(4+2\delta)
\|A-B\|.
\end{equation}
Since $\hat{B}$ and $G$ both have rank $\le|S|$, the difference $\hat
{B}-G$ has rank at most~$2|S|$. Using this and \eqref{ab3}, we have
%
\begin{equation}
\label{ab4}
\|\hat{B}-G\|_F \le\sqrt{2|S|} \|\hat{B}-G\| \le(4+2
\delta)\sqrt {2|S|} \|A-B\|.
\end{equation}
Next, observe that by \eqref{tauineq},
%
\begin{equation}
\label{ab5}
\hspace*{9pt}\|B-G\|_F^2 = \sum
_{i\notin S} \tau_i^2 \le(2+\delta)\|A-B\|
\sum_{i\notin S} \tau_i \le(2+\delta)\|A-B\|\|B
\|_*.
\end{equation}
Combining \eqref{ab4} and \eqref{ab5}, we have
%
\begin{eqnarray}
\|\hat{B}-B\|_F  & \le & \|\hat{B}-G\|_F + \|B-G
\|_F
\nonumber
\\[-8pt]
\label{abmain}\\[-8pt]
\nonumber
&\le & (4+2\delta)\sqrt{2|S|} \|A-B\| + \bigl((2+\delta)\|A-B\|\|B\| _*
\bigr)^{1/2}.
\end{eqnarray}
Next,\vspace*{-2pt} note that by \eqref{tau2},
\[
\|B\|_* \ge\sum_{i\in S} \tau_i \ge
\delta|S|\|A-B\|,
\]
and thus\vspace*{-2pt}
%
\begin{equation}
\label{sineq}
|S|\le\frac{\|B\|_*}{\delta\|A-B\|}.
\end{equation}
Combining \eqref{abmain} and \eqref{sineq}, the proof is complete.
\end{pf}

\subsubsection*{Finishing the proof of Theorem~\texorpdfstring{\protect\ref{mainest}}{1.1}}
We will prove the theorem only for the asymmetric model. The only
difference in the proofs for the symmetric model and the skew-symmetric
model is that we need to use the symmetric and skew-symmetric parts of
Theorem~\ref{normthm} instead of the asymmetric part.

Throughout this proof, $C(\ep)$ will denote any constant that depends
only on $\ep$ and $\eta$, and $C$ and $c$ will denote constants that
depend only on $\eta$. The values of $C(\ep)$, $C$ and $c$ may change
from line to line or even within a line. We will use the fact that $\eta
\in(0,1)$ without mention on many occasions.

Note that for\vspace*{-2pt} all $i$ and $j$,
\[
\ee(y_{ij})= pm_{ij}
\]
and\vspace*{-2pt}
%
\begin{equation}
\label{varyij}
\var(y_{ij}) \le\ee\bigl(y_{ij}^2
\bigr) = p\ee\bigl(x_{ij}^2\bigr)\le p.
\end{equation}
Let $\hp$ be the proportion of\vspace*{-2pt} observed entries. 
Define two events $E_1$ and $E_2$ as
\begin{eqnarray*}
E_1 &:= &\bigl\{\|Y-pM\|\le(2+\eta/2)\sqrt{np}\bigr\},
\\
E_2 &:=& \bigl\{|\hp-p|\le\eta p/20\bigr\}.
\end{eqnarray*}
By\vspace*{-2pt} Theorem~\ref{normthm},
%
\begin{equation}
\label{pe1}
\pp(E_1)\ge1- C(\ep)e^{- cnp}.
\end{equation}
By Bernstein's inequality (Theorem~\ref{bern}), for any $t\ge0$,
\[
\pp\bigl(|\hp-p|\ge t\bigr) \le2\exp \biggl(-\frac{3mn t^2}{6p(1-p) + 2t} \biggr).
\]
In particular,
%
\begin{equation}
\label{pe2}
\pp(E_2)\ge1- 2e^{-cmnp}.
\end{equation}
Let $\delta$ be defined by the relation
\[
(1+\delta)\|Y-pM\| = (2+\eta) \sqrt{n\hp}.
\]
If $E_1$ and $E_2$ both happen, then
\[
1+\delta\ge\frac{(2+\eta)\sqrt{n\hp}}{(2+\eta/2)\sqrt{np}}\ge\frac
{(2+\eta)\sqrt{(1-\eta/20)np}}{(2+\eta/2)\sqrt{np}}\ge1 + \eta/5.
\]
Let $K(\delta)$ be the constant in the statement of Lemma~\ref{estlmm}.
It is easy to see that there is a constant $C$ depending only on $\eta$
such that if $\delta\ge\eta/5$, then $K(\delta)\le C \sqrt{1+\delta
}$. Therefore, by Lemma~\ref{estlmm}, if $E_1$ and $E_2$ both happen, then
%
\begin{eqnarray}
\|\hp W-pM\|_F^2 &\le  & C(1+\delta) \|Y-pM\|\|pM\|_*
\nonumber
\\
\label{wpm} &\le&  C\sqrt{n\hp}\|pM\|_*
\\
&\le &  Cn^{1/2}p^{3/2}\|M\|_*.\nonumber
\end{eqnarray}
By the definition of $\hat{M}$, it is obvious that $|\hat
{m}_{ij}-m_{ij}|\le|w_{ij}-m_{ij}|$ for all $i$ and $j$. Together with
\eqref{wpm}, this shows that under $E_1\cap E_2$,
\begin{eqnarray*}
p^2\|\hat{M}-M\|_F^2 &\le&  p^2
\|W-M\|_F^2
\\
&\le&  C\hp^2\|W-M\|_F^2
\\
&\le &  C\|\hp W - pM\|_F^2 + C (\hp-p)^2\|M
\|_F^2
\\
&\le & Cn^{1/2}p^{3/2}\|M\|_*+ C(\hp-p)^2mn.
\end{eqnarray*}
Note that $\ee(\hp-p)^2 = p(1-p)/mn$ and that $\|\hat{M}-M\|^2_F\le
4mn$. Thus, by \eqref{pe1} and \eqref{pe2},
\begin{eqnarray*}
\ee\|\hat{M}-M\|_F^2 &\le & Cn^{1/2}
p^{-1/2}\|M\|_* + Cp^{-1} + Cmn \bigl(1-\pp(E_1
\cap E_2)\bigr)
\\
&\le &  Cn^{1/2}p^{-1/2} \|M\|_* + Cp^{-1} + C(\ep)mn
e^{-cnp}.
\end{eqnarray*}
%
Dividing throughout by $mn$, we arrive at the inequality
%
\begin{equation}
\label{mmbd1}
\frac{1}{mn} \ee\|\hat{M}-M\|_F^2
\le\frac{C\|M\|_*}{m\sqrt{np}} + \frac{C}{np} + C(\ep)e^{-cnp}.
\end{equation}
The next goal is to show that
%
\begin{equation}
\label{mmbd2}
\frac{1}{mn}\ee\|\hat{M}-M\|_F^2 \le
\frac{C\|M\|_*^2}{mn} + C(\ep)e^{-cnp}.
\end{equation}
First, suppose that $\|M\|_*> \eta\sqrt{n/p}/20$. Then
\[
\frac{\|M\|_*}{m\sqrt{np}}+\frac{1}{np} \le\frac{C\|M\|_*^2}{mn},
\]
and so \eqref{mmbd2} follows from \eqref{mmbd1}. Therefore, assume that
$\|M\|_* \le\eta\sqrt{n/p}/20$. Then in particular, $\|M\|\le\eta\sqrt
{n/p}/20$. Therefore, if $E_1\cap E_2$ happens, then
\begin{eqnarray*}
\|Y\|&\le & \|Y-pM\|+\|pM\|
\\
&\le & (2+\eta/2+\eta/20)\sqrt{np}
\\
&\le & \frac{(2+11\eta/20)\sqrt{n\hp}}{1-\eta/20}\le(2+13\eta/20)\sqrt {n\hp}.
\end{eqnarray*}
This implies that there is no singular value of $Y$ that exceeds
$(2+\eta)\sqrt{n\hp}$, and therefore $\hat{M}=0$. Consequently,
\[
\|\hat{M}-M\|_F^2 =\|M\|_F^2
\le\|M\|_*^2.
\]
Thus, if $\|M\|_*\le\eta\sqrt{n/p}/20$, then by \eqref{pe1} and \eqref{pe2},
\[
\frac{1}{mn}\ee\|\hat{M}-M\|_F^2 \le
\frac{\|M\|_*^2}{mn} + C\bigl(1-\pp (E_1\cap E_2)\bigr) \le
\frac{\|M\|_*^2}{mn} + C(\ep)e^{-cnp}.
\]
Combining the above steps and observing that $\operatorname{MSE}(\hat{M})\le
1$ due to the boundedness of the entries of $M$ and $\hat{M}$, we get
\[
\operatorname{MSE}(\hat{M}) \le C\min \biggl\{\frac{\|M\|_*}{m\sqrt{np}} + \frac{1}{np} ,
\frac{\|M\|_*^2}{mn} , 1 \biggr\} + C(\ep)e^{-cnp}.
\]
To remove the $1/np$ term, note that if that term indeed matters, then
we are in a situation where
\[
\frac{\|M\|_*}{m\sqrt{np}}\le\frac{\|M\|_*^2}{mn}.
\]
But this inequality, on the other hand, implies that
\[
\frac{\|M\|_*}{m\sqrt{np}}\ge\frac{1}{mp}\ge\frac{1}{np}.
\]
Therefore, the $1/np$ term can be removed from the above bound.
This completes the proof of Theorem~\ref{mainest} if no nontrivial
bound on $\var(x_{ij})$ is known.

If $\sigma^2\le1$ is a known constant such that $\var(x_{ij})\le\sigma
^2$ for all $i, j$, then the estimate \eqref{varyij} may be improved to
\[
\var(y_{ij})= p\var(x_{ij}) + p(1-p) m_{ij}^2
\le\max_{(a,b)\in R} \bigl(pb + p(1-p) a\bigr),
\]
where $R$ is the quadrilateral region
\[
\bigl\{(a,b)\dvtx  0\le a\le1, 0\le b\le\sigma^2, 0\le a+b\le1\bigr\}.
\]
The maximum must be attained at one of the four vertices of $R$. An
easy verification shows that the maximum is always attained at the
vertex $(1-\sigma^2, \sigma^2)$, which gives the upper bound
\[
\var(y_{ij})\le q := p\sigma^2 +p(1-p) \bigl(1-
\sigma^2\bigr) .
\]
This allows us to replace the threshold $(2+\eta)\sqrt{n\hat{p}}$ by
$(2+\eta)\sqrt{n\hat{q}}$, where $\hat{q}= \hat{p}\sigma^2 + \hat
{p}(1-\hat{p}) (1-\sigma^2)$. As before, we need that $q \ge n^{-1+\ep
}$. The rest of the proof goes through with the following
modifications: Replace $\sqrt{np}$ by $\sqrt{nq}$ in the definition of
$E_1$, keep $E_2$ the same, and define an event $E_3 = \{|\hat{q}-q|\le
\eta q/20\}$. By Theorem~\ref{normthm}, $\pp(E_1)\ge1-C(\ep)e^{-cnq}$,
$\pp(E_2)\ge1-2e^{-cmnp}\ge1-2e^{-cmn q}$, and $\pp(E_3) \ge
1-2e^{-cmnq}$ since $|\hat{q}-q|\le|\hat{p}-p|$ and $q\ge p(1-p)$ and,
therefore,
\[
\pp\bigl(E_3^c\bigr) \le\pp\bigl(|\hat{p}-p|> \eta q/20\bigr)
\le2\exp \biggl(-\frac{cmn
q^2}{6p(1-p) + \eta q/10} \biggr)\le2e^{-cmnq}.
\]
If $E_1\cap E_2 \cap E_3$ happens, then the subsequent steps remain the
same, but with some suitable modifications that replace the term $\|M\|
_*/(m\sqrt{np})$ by the improved term $\|M\|_* \sqrt{q}/(m\sqrt{n} p)$.

\subsection{Proof of Theorem~\texorpdfstring{\protect\ref{minimaxthm}}{1.2} (Minimax optimality)}

Throughout this proof, $C$~will denote any positive universal constant,
whose value may change from line to line.

Take any $\delta\in[0, m\sqrt{n}]$ and let $\theta:= \delta/(m\sqrt
{n})$. We will first work out the proof under the assumption that
$p<1/2$. Under this assumption, three situations are considered. First,
suppose that
%
\begin{equation}
\label{thetap1}
\theta/\sqrt{p}\le1 \quad \mbox{and}\quad  m\theta\sqrt{p}\ge1.
\end{equation}
Let $k:=[m\theta\sqrt{p}]$. Clearly, $k\le m$. Let $M$ be an $m\times
n$ random matrix whose first $k$ rows consist of i.i.d.
$\operatorname{Uniform}[-1,1]$ random variables, and copy this block $[1/p]$ times.
This takes care of $k[1/p]$ rows. [This is okay, since $k/p\le m\theta
/\sqrt{p}\le m$ by \eqref{thetap1}.] Declare the remaining rows, if
any, to be zero. Then note that $M$ has rank $\le k\le m\theta\sqrt
{p}$. Therefore, by inequality~\eqref{csineq0},
\[
\|M\|_*  \le(m\theta\sqrt{p})^{1/2} \|M\|_F \le(m\theta
\sqrt {p})^{1/2} (mn\theta/\sqrt{p})^{1/2} = m \sqrt{n} \theta.
\]
Let $X=M$. Let $D$ be our data, that is, the observed values of $X$.
One can imagine $D$ as a matrix whose $(i,j)$th entry is $x_{ij}$ if
$x_{ij}$ is observed, and a question mark if $x_{ij}$ is unobserved.
For any $(i,j)$ belonging to the nonzero portion of the matrix $M$,
$M$ contains $[1/p]$ copies of $m_{ij}$. Since the $X$-value at the
location of each copy is observed with probability $p$, independent of
the other copies, and $p<1/2$, therefore, the chance that none of these
copies are observed is bounded below by a positive universal constant.
If none of the copies are observed, then the data contains no
information about $m_{ij}$. Using this, it is not difficult to write
down a formal argument that shows
\[
\ee\bigl(\var(m_{ij}\vert D)\bigr) \ge C.
\]
On the other hand, since $\tilde{m}_{ij}$ is a function of $D$, the
definition of variance implies that
\[
\ee\bigl((\tilde{m}_{ij}-m_{ij})^2\vert D\bigr)
\ge\var(m_{ij}\vert D).
\]
Combining the last two displays, we see that
%
\begin{eqnarray}
\ee\|\tilde{M}-M\|_F^2 &\ge & \sum
_{i=1}^{k[1/p]} \sum_{j=1}^n
\ee(\tilde {m}_{ij}-m_{ij})^2
\nonumber
\\[-8pt]
\label{mmf}\\[-8pt]
\nonumber
&\ge &  Ck[1/p]n\ge\frac{Cmn\theta}{\sqrt{p}}.
\end{eqnarray}
The argument that led to the above lower bound is a typical example of
the classical Bayesian argument for obtaining minimax lower bounds, and
will henceforth be referred to as the ``standard minimax argument'' to
avoid repetition of details.

Next, assume that
%
\begin{equation}
\label{thetap2}
\theta/\sqrt{p}\le1   \quad \mbox{and}\quad  m\theta\sqrt{p}< 1.
\end{equation}
Let $M$ be an $m\times n$ matrix whose first row consists of i.i.d.
random variables uniformly distributed over the interval $[-m\theta\sqrt
{p},m\theta\sqrt{p}]$, and this row is copied $[1/p]$ times, and all
other rows are zero. Then $M$ has rank $\le 1$, and therefore by
inequality \eqref{csineq0},
\[
\|M\|_* \le\|M\|_F \le m\theta\sqrt{p} \sqrt{\frac{n}{p}} =
m\theta \sqrt{n}.
\]
On the other hand, a standard minimax argument as before implies that
for any estimator $\tilde{M}$,
\[
\ee\|\tilde{M}-M\|_F^2 \ge(m\theta\sqrt{p})^2
\frac{n}{p} = nm^2 \theta^2.
\]
In particular, under \eqref{thetap2}, there exists $M$ with $\|M\|_*\le
\delta$ such that
\[
\operatorname{MSE}(\hat{M})\ge\frac{C\delta^2}{mn}.
\]
Finally, suppose that
%
\begin{equation}
\label{thetap3}
\theta/\sqrt{p}> 1.
\end{equation}
Let $M$ be an $m\times n$ matrix whose first $[mp]$ rows consist of
i.i.d.  random variables uniformly distributed over $[-1,1]$, and this
block is copied $[1/p]$ times. Then the rank of $M$ is $\le[mp]$, and
so by \eqref{thetap3} and \eqref{csineq0},
\[
\|M\|_*\le\sqrt{mp} \|M\|_F\le m \sqrt{np}\le\theta m\sqrt{n}.
\]
Again by a standard minimax argument, it is easy to conclude that for
any estimator $\tilde{M}$, there exists $M$ with $\|M\|_*\le\delta$
such that
\[
\operatorname{MSE}(\tilde{M}) \ge C.
\]
This completes the proof when $p< 1/2$. Next, suppose that $p\ge1/2$.
The only place where the assumption $p< 1/2$ was used previously was
for proving that $\ee(\var(m_{ij}|D)) \ge C$. This can be easily taken
care of by inserting some randomness into the data matrix $X$, as
follows. First, replace $M$ by $\frac{1}{2}M$ in all three cases above.
This retains the condition $\|M\|_*\le\delta$. Given $M$, let $X$ be
the data matrix whose $(i,j)$th entry $x_{ij}$ is uniformly distributed
over the interval $[m_{ij}-1/2, m_{ij}+1/2]$, whenever $(i,j)$ is the
``main block'' of $M$; and this value of $x_{ij}$ is copied $[1/p]$ times
in the appropriate places.

Since the entries of $M$ are now guaranteed to be in $[-1/2,1/2]$, this
ensures that the entries of $X$ are in $[-1,1]$. Now note that even if
$x_{ij}$ or one of its copies is observed, it gives only limited
information about $m_{ij}$. In particular, it is easy to prove that $\ee
(\var(m_{ij}|D))\ge C$ and complete the proof as before.

This complete the proof of Theorem~\ref{minimaxthm} for the asymmetric
model. For the symmetric model, simply observe that the singular values
of any square matrix $M$ are the same as those of the symmetric matrix
\[
\left[
\matrix{
0 & M
\cr
M^T & 0}
 \right]
\]
with multiplicity doubled. It is now clear how the minimax arguments
for the asymmetric model may be carried over to the symmetric case by
considering the same Bayesian models for $M$ and working with the
corresponding\vspace*{1pt} symmetrized matrices. For the skew-symmetric case,
replace the $M^T$ by $-M^T$ in the above matrix.

\subsection{Proof of Theorem~\texorpdfstring{\protect\ref{impossible}}{1.3} (Impossibility of error
estimation)}
Suppose that a good procedure P exists. By the definition of
nontriviality of the estimator $\tilde{M}$, there exists a sequence of
parameter matrices $M_n$ and data matrices $X_n$ such that
%
\begin{equation}
\label{imposs1}
\operatorname{MSE}\bigl(\hat{M}_n^{\mathrm{Trv}}\bigr) \not\to0 \qquad \mbox{as } n\ra \infty,
\end{equation}
but
%
\begin{equation}
\label{imposs2} \lim_{n\ra\infty} \operatorname{MSE}(\tilde{M}_n)=0.
\end{equation}
Then by the definition of goodness,
\[
\widehat{\operatorname{MSE}}_{\mathrm{P}}(\tilde{M}_n) \to0  \qquad\mbox{in probability as }n\to\infty.
\]
%
Suppose, without loss of generality, that all the data matrices are
defined on the same probability space. Then taking a subsequence if
necessary, we may assume that in addition to \eqref{imposs1} and \eqref
{imposs2}, we also have
%
\begin{equation}
\label{imposs3}
\pp \Bigl(\lim_{n\ra\infty} \widehat{\operatorname{MSE}}_{\mathrm{P}}(\tilde {M}_{n}) = 0 \Bigr) = 1.
\end{equation}
Let $M_n' := X_n$ and $X_n' := X_n$ for all $n$. Consider $M_n'$ as a
(random) parameter matrix and $X_n'$ as its data matrix. Given $M_n'$,
the expected value of $X_n'$ is $M_n'$; so it is okay to treat $M_n'$
as a parameter matrix and $X_n'$ as its data matrix. We will denote the
estimate of $M_n'$ constructed using $X_n'$ as $\tilde{M}_n'$. Note
that since $M_n'$ is random, the mean squared error of $\tilde{M}_n'$
is a random variable.

Now, since the estimator $\tilde{M}_n'$ is computed using the data
matrix only, and $X_n'=X_n$, it is clear that $\tilde{M}_n' = \tilde
{M}_n$. There is no randomness in $\tilde{M}_n'$ when $M_n'$ is given,
since $X_n'=M_n'$. Thus, if $r_n$ and $c_n$ denote the number of rows
and columns of $M_n$, then
\begin{eqnarray*}
\operatorname{MSE}\bigl(\tilde{M}_n'\bigr) &=&
\frac{1}{r_nc_n} \bigl\|\tilde{M}_n' - M_n'
\bigr\| _F^2
\\
&=& \frac{1}{r_nc_n} \|\tilde{M}_n - X_n
\|_F^2
\\
&= &\frac{1}{r_nc_n} \bigl\|\tilde{M}_n - \hat{M}^{\mathrm{Trv}}_n
\bigr\|_F^2
\\
&\ge &\frac{1}{2r_n c_n} \bigl\|\hat{M}^{\mathrm{Trv}}_n - M_n
\bigr\|_F^2 - \frac
{1}{r_n c_n} \|\tilde{M}_n -
M_n\|_F^2,
\end{eqnarray*}
where the last step follows from the inequality $(a+b)^2 \le2a^2 +
2b^2$ and the triangle inequality for the Frobenius norm. Taking
expectation on both sides gives
\[
\ee\bigl(\operatorname{MSE}\bigl(\tilde{M}_n'\bigr)\bigr)
\ge\tfrac{1}{2} \operatorname{MSE}\bigl(\hat {M}_n^{\mathrm{Trv}}
\bigr) - \operatorname{MSE}(\tilde{M}_n).
\]
Therefore, by \eqref{imposs1} and \eqref{imposs2},
\[
\ee\bigl(\operatorname{MSE}\bigl(\tilde{M}_n'\bigr)\bigr)
\not\to0 \qquad \mbox{as }n\ra\infty.
\]
In particular, since mean squared errors are uniformly bounded by $1$,
%
\begin{equation}
\label{posprob}
\pp \bigl(\operatorname{MSE}\bigl(\tilde{M}_n'
\bigr) \not\to0  \mbox{ as }n\ra\infty\bigr) > 0.
\end{equation}
Again since $\widehat{\operatorname{MSE}}_{\mathrm{P}}$ is computed using the
data matrix only, therefore, for all~$n$,
\[
\widehat{\operatorname{MSE}}_{\mathrm{P}}(\tilde{M}_n) = \widehat{
\mathrm {MSE}}_{\mathrm{P}}\bigl(\tilde{M}_n'\bigr).
\]
Therefore, by \eqref{imposs3},
%
\begin{equation}
\label{posprob2}
\pp \Bigl(\lim_{n\ra\infty}\widehat{
\operatorname{MSE}}_{\mathrm{P}}\bigl(\tilde {M}_n'\bigr)= 0
\Bigr) =1.
\end{equation}
Equations \eqref{posprob} and \eqref{posprob2} demonstrate the
existence of a sequence of parameter matrices $M_n'$ and data matrices
$X_n'$ such that $\operatorname{MSE}(\tilde{M}'_n) \not\to0$ but $\widehat
{\operatorname{MSE}}_{\mathrm{P}}(\tilde{M}_n') \to0$ in probability. This
contradicts the goodness of $\widehat{\operatorname{MSE}}_{\mathrm{P}}$.

\subsection{Proof of Theorem~\texorpdfstring{\protect\ref{lowrankthm}}{2.1} (Upper bound for low rank
matrix estimation)}

Inequality \eqref{csineq0} implies that
\[
\|M\|_*\le \sqrt{\rank(M)} \|M\|_F\le\sqrt{rmn}.
\]
The result now follows from Theorem~\ref{mainest}.

\subsection{Proof of Theorem~\texorpdfstring{\protect\ref{lowranklow}}{2.2} (Lower bound for low rank
matrix estimation)}

Let $M$ be an $m\times n$ random matrix whose first $r$ rows consist of
i.i.d. $\operatorname{Uniform}[-1,1]$ random variables, and copy this block $[m/r]$
times. Declare the remaining rows, if any, to be zero. Then note that
$M$ has rank${}\le r$.

Let $D$ be our data, that is, the observed values of $M$. One can
imagine $D$ as a matrix whose $(i,j)$th entry is $m_{ij}$ if $m_{ij}$
is observed, and a question mark if $m_{ij}$ is unobserved. For any
$(i,j)$ belonging to the nonzero portion of the matrix $M$, $M$
contains $[m/r]$ copies of $m_{ij}$. Since the $M$-value at the
location of each copy is observed with probability $p$, independent of
the other copies, the chance that none of these copies are observed is
equal to $(1-p)^{[m/r]}$. If none of the copies are observed, then the
data contains no information about $m_{ij}$. Using this, it is not
difficult to write down a formal argument that shows
\[
\ee\bigl(\var(m_{ij}\vert D)\bigr) \ge C(1-p)^{[m/r]},
\]
where $C$ is some universal constant.
On the other hand, since $\tilde{m}_{ij}$ is a function of~$D$, the
definition of variance implies that
\[
\ee\bigl((\tilde{m}_{ij}-m_{ij})^2\vert D\bigr)
\ge\var(m_{ij}\vert D).
\]
Combining the last two displays, we see that
\[
\ee\|\tilde{M}-M\|_F^2  \ge Cmn (1-p)^{[m/r]}.
\]
This completes the proof.

\subsection{Proof of Theorem~\texorpdfstring{\protect\ref{blockthm}}{2.3} (Block model estimation)}
If two vertices $i$ and $j$ are in the same block, then the $i$th and
$j$th rows of $M$ are identical. Therefore, $M$ has at most $k$
distinct rows and so the rank of $M$ is $\le k$. An application of
Theorem~\ref{lowrankthm} completes the proof.

\subsection{Proofs of Theorems \texorpdfstring{\protect\ref{compact1}}{2.4} 
and \texorpdfstring{\protect\ref{compact2}}{2.5}
(Distance matrix estimation)}
The proofs of Theorems \ref{compact1} and \ref{compact2} follow from a
more general lemma that will also be useful later for other purposes.
Suppose that $S = \{x_1,\ldots,x_n\}$ is a finite set and $f\dvtx S\times S
\ra[-1,1]$ is an arbitrary function. Suppose that for each $\delta>0$,
there exists a partition $\cp(\delta)$ of $S$ such that whenever
$x,y,x',y'$ are four points in $S$ such that $x,x'\in P$ for some $P\in
\cp(\delta)$ and $y,y'\in Q$ for some $Q\in\cp(\delta)$, then
$|f(x,y)-f(x', y')|\le\delta$. Let $M$ be the $n\times n$ matrix whose
$(i,j)$th element is $f(x_i, x_j)$.

\begin{lmm}\label{function}
In the above setting,
\[
\operatorname{MSE}(\hat{M}) \le C\inf_{\delta> 0} \min \biggl\{
\frac{\delta
+ \sqrt{|\cp(\delta)|/n}}{\sqrt{p}} , 1 \biggr\} + C(\ep)e^{-cnp},
\]
where $C$ and $c$ depend only on $\eta$, and $C(\ep)$ depends only on
$\ep$ and $\eta$.
\end{lmm}

\begin{pf}
Fix some $\delta>0$. Let $T$ be a subset of $S$ consisting of exactly
one point from each member of $\cp(\delta)$. For each $x\in S$, let
$p(x)$ be the unique element of $T$ such that $x$ and $p(x)$ belong to
the same element of $\cp(\delta)$. Let $N$ be the matrix whose
$(i,j)$th element is $f(p(x_i), p(x_j))$. Then
\[
\|M-N\|_F^2 = \sum_{i,j=1}^n
\bigl(f(x_i, x_j)-f\bigl(p(x_i),
p(x_j)\bigr)\bigr)^2 \le n^2
\delta^2.
\]
By the triangle inequality for the nuclear norm, the inequality \eqref
{csineq0} and the above inequality,
\begin{eqnarray*}
\|M\|_* &\le & \|M-N\|_* + \|N\|_*
\\
&\le & \sqrt{n}\|M-N\|_F + \|N\|_*
\\
&\le &  n^{3/2} \delta+ \|N\|_*.
\end{eqnarray*}
Now, if $x_i$ and $x_j$ belong to the same element of $\cp(\delta)$,
then $p(x_i)=p(x_j)$, and hence the $i$th and $j$th rows of $N$ are
identical. This shows that $N$ has at most $|\cp(\delta)|$ distinct
rows and, therefore, has rank${}\le|\cp(\delta)|$. Therefore, by the
inequality \eqref{csineq0},
\[
\|N\|_*\le\sqrt{\bigl|\cp(\delta)\bigr|} \|N\|_F \le\sqrt{\bigl|\cp(\delta)\bigr|} n.
\]
The proof is completed by applying Theorem~\ref{mainest}.
\end{pf}

Using Lemma~\ref{function}, it is easy to prove Theorems \ref{compact1}
and \ref{compact2}.
\begin{pf*}{Proof of Theorem~\ref{compact2}}
Let all notation be as in Theorem~\ref{compact2}. To apply Lemma~\ref
{function}, let $S$ be the set $\{x_1,\ldots, x_n\}$. From the
definition of $N(\delta)$, it is easy to see that there is a partition
$\cp(\delta)$ of $S$ of size $\le N(\delta/4)$, such that any two
points belonging to the same element of the partition are at distance
$\le\delta/2$ from each other. Consequently, if $x,x'\in P$ and
$y,y'\in Q$ for some $P, Q\in\cp(\delta)$, then by the triangle
inequality for the metric $d$,
\begin{eqnarray*}
\bigl|d(x,y) - d\bigl(x', y'\bigr)\bigr| &\le & \bigl|d(x,y)-d
\bigl(x', y\bigr)\bigr| + \bigl|d\bigl(x', y\bigr)-d
\bigl(x', y'\bigr)\bigr|
\\
&\le &  d\bigl(x,x'\bigr) + d\bigl(y,y'\bigr) \le
\delta.
\end{eqnarray*}
Putting $f=d$ in Lemma~\ref{function}, the proof is complete.
\end{pf*}

\begin{pf*}{Proof of Theorem~\ref{compact1}}
Since $K$ is compact, there exists a finite number $N(\delta)$ for each
$\delta> 0$ such that $K$ may be covered by $N(\delta)$ open $d$-balls
of radius~$\delta$. By Theorem~\ref{compact2}, this shows that for any
sequence $\delta_n$ decreasing to zero,
\[
\operatorname{MSE}(\hat{M}) \le C\min \biggl\{\frac{\delta_n + \sqrt{N(\delta
_n/4)/n}}{\sqrt{p}} , 1 \biggr\} + C(
\ep)e^{-cnp}.
\]
To complete the proof, choose $\delta_n$ going to zero so slowly that
$N(\delta_n/4)=o(n)$ as $n \ra\infty$.
\end{pf*}

\subsection{Proof of Theorem~\texorpdfstring{\protect\ref{latent1}}{2.6} (Latent space models:
General case)}
We will apply Lemma~\ref{function}. Let $S$ be the set $\{\beta_1,\ldots
, \beta_n\}$. Since $f$ is continuous on $K$ and $K$ is compact, $f$
must be uniformly continuous. This shows that for each $\delta> 0$ we
can find a partition $\cp(\delta)$ of $S$ satisfying the condition
required for Lemma~\ref{function}, such that the size of $\cp(\delta)$
may be bounded by a constant $N(\delta)$ depending only on $K$, $k$,
$f$ and $\delta$. Choosing $\delta_n \ra0$ slowly enough so that
$N(\delta_n/4)=o(n)$ and applying Lemma~\ref{function} completes the proof.

\subsection{Proof of Theorem~\texorpdfstring{\protect\ref{latent2}}{2.7} (Latent space models:
Lipschitz functions)}
Let $S=\{\beta_1,\ldots, \beta_n\}$. Take any $\delta> 0$. From the
Lipschitzness condition, it is easy to see that we can find a partition
$\cp(\delta)$ of $S$ whose size may be bounded by $C(K,k,L)\delta
^{-k}$, where $C(K,k,L)$ depends only on $K$, $k$ and $L$. Choosing
$\delta= n^{-1/(k+2)}$ and applying Lemma~\ref{function} completes the
proof. Note that the exponential term need not appear since the main
term is bounded below by a\vspace*{1.5pt} positive constant if $p < n^{-2/(k+2)}$.

\subsection{Proof of Theorem~\texorpdfstring{\protect\ref{pdthm}}{2.8} (Upper bound for positive
definite matrix estimation)}
Since $M$ is positive semi-definite, $\|M\|_*=\tr(M)$. Since the
entries of $M$ are bounded by $1$, $\tr(M)\le n$. The proof now follows
from an application of Theorem~\ref{mainest}.

\subsection{Proof of Theorem~\texorpdfstring{\protect\ref{pdlow}}{2.9} (Lower bound for positive
definite matrix estimation)}
Throughout this proof, $C$ will denote any positive universal constant,
whose value may change from line to line.

Let $U_1,\ldots, U_n$ be i.i.d. $\operatorname{Uniform}[0,1]$ random variables. Let
$M$ be the random matrix whose $(i,j)$th element $m_{ij}$ is equal to
$U_iU_j$ if $i\ne j$ and $1$ if $i=j$. It is easy to verify that $M$ is
a correlation matrix. Suppose that we observe each element of $M$ on
and above the diagonal with probability $p$, independent of each other.
Let $D$ be our data, represented as follows: $D$ is a matrix whose
$(i,j)$th element is $m_{ij}$ if the element is observed, and a
question mark otherwise.

Now, the probability that no element from the $i$th row and the $i$th
column is observed is exactly equal to $(1-p)^n$. If we do not observe
any element from the $i$th row and $i$th column, we have no information
about the value of $U_i$. From this, it is not difficult to write down
a formal argument to prove that for any $j\ne i$,
\[
\var\bigl(m_{ij} \vert D, (U_k)_{k\ne i}\bigr) =
U_j^2 \var\bigl(U_i \vert D,
(U_k)_{k\ne i}\bigr)\ge C(1-p)^nU_j^2.
\]
If $\tilde{M}$ is any estimator, then $\tilde{m}_{ij}$ is a function of
$D$. Therefore, by the above inequality and the definition of variance,
\[
\ee\bigl((\tilde{m}_{ij}-m_{ij})^2\vert D,
(U_k)_{k\ne i}\bigr) \ge C(1-p)^nU_j^2,
\]
and thus
\[
\ee(\tilde{m}_{ij}-m_{ij})^2 \ge
C(1-p)^n.
\]
Since this is true for all $i\ne j$, the proof is complete.

\subsection{Proof of Theorem~\texorpdfstring{\protect\ref{graphonthm}}{2.10} (Graphon estimation)}
Here, all entries of the adjacency matrix are visible, so $p=1$. Define
a sequence of functions $f_1, f_2,\ldots$ according the following
standard construction. For each $k$, let $\cp_k$ be the $k$th level
dyadic partition of $[0,1)^2$, that is, the partition of the unit
square into sets of the form $[(i-1)/2^k, i/2^k) \times[(j-1)/2^k,
j/2^k)$. Let $f_k$ be the function that is equal to the average value
of $f$ within each square of the partition $\cp_k$. If $\mf_k$ denotes
the sigma-algebra of sets generated by the partition $\cp_k$, then the
sequence $f_k$ is a martingale with respect to the filtration $\mf_k$.
Moreover, $f_k = \ee(f|\mf_k)$. Finally, observe that the sequence
$f_k$ is uniformly bounded in $L^2$. Combining all these observations,
it is evident that $f_k \ra f$ in $L^2$.

Now fix some $\ep>0$ and an integer $n$. Take a large enough $k = k(\ep
)$ such that $\|f-f_k\|_{L^2} \le\ep$. Let $N$ be the $n\times n$
matrix whose $(i,j)$th element is $f_k(U_i, U_j)$. Then
%
\begin{eqnarray}
\ee\|M - N\|_F^2 &=&  \sum_{i,j=1}^n
\ee\bigl(f(U_i, U_j) - f_k(U_i,
U_j)\bigr)^2
\nonumber
\\
\label{mstar1} &\le&  n + n^2\ee\bigl(f(U_1, U_2)-f_k(U_1,
U_2)\bigr)^2
\\
&=&  n + n^2\|f-f_k\|_{L^2}^2\le n
+n^2\ep^2.\nonumber
\end{eqnarray}
Now note that if $U_i$ and $U_j$ belong to the same dyadic interval
$[r/2^k, (r+1)/2^k)$, then the $i$th and $j$th rows of $N$ are
identical. Hence, $N$ has at most $2^k$ distinct rows, and therefore
has rank $\le2^k$. Therefore, by \eqref{csineq0},
\[
\|N\|_*\le2^{k/2} \|N\|_F \le2^{k/2} n.
\]
Consequently, by the inequality \eqref{csineq0},
%
\begin{eqnarray}
\|M\|_* &\le & \|M-N\|_* + \|N\|_*
\nonumber
\\[-8pt]
\label{mstar2}\\[-8pt]
\nonumber
&\le & \sqrt{n} \|M-N\|_F + 2^{k/2}n.
\end{eqnarray}
Combining \eqref{mstar1} and \eqref{mstar2} gives
\[
\ee\|M\|_* \le\bigl(2^{k/2}+1\bigr) n + n^{3/2}\ep.
\]
Choosing a sequence $\ep_n$ going to zero so slowly that $2^{k(\ep
_n)/2} = o(n^{-1/2})$, we can now apply Theorem~\ref{mainest} to
complete the proof.

\subsection{Proof of Theorem~\texorpdfstring{\protect\ref{bradleythm}}{2.11} (Bradley--Terry models)}
Throughout the proof $C$ will denote any constant that depends only on
$\eta$, whose value may change from line to line.

Recall that the definition of the skew-symmetric model stipulates that
$X-M$ is skew-symmetric, which is true for the nonparametric
Bradley--Terry model. There is nothing to prove if $p< n^{-2/3}$, so
assume that $p\ge n^{-2/3}$. This allows us to drop the exponential
term in Theorem~\ref{mainest} and conclude that
%
\begin{equation}
\label{bt0}
\operatorname{MSE}(\hat{M})\le\frac{C\|M\|_*}{n^{3/2}\sqrt{p}}.
\end{equation}
Let $k$ be an integer less than $n$, to determined later. For each $i$, let
\[
t_i := \sum_{j=1}^n
p_{ij}.
\]
Note that each $t_i$ belongs to the interval $[0,n]$. For $l=1,\ldots,
k$, let $T_l$ be the set of all $i$ such that $t_i\in[n(l-1)/k,
nl/k)$. Additionally, if $t_i=n$, put $i$ in $T_k$.

For each $l$, let $r(l)$ be a distinguished element of $T_l$. For each
$1\le i,j\le n$, if $i\in T_l$ and $j\in T_m$, let $n_{ij} :=
p_{r(l)j}$. Let $N$ be the matrix whose $(i,j)$th element is~$n_{ij}$.
Note that if $i,i'\in T_l$ for some $l$, then $n_{ij} = n_{i'j}$ for
all~$j$. In particular, $N$ has at most $k$ distinct rows and therefore
has rank $\le k$. Thus, by inequality \eqref{csineq0},
%
\begin{equation}
\label{bt1}
\|N\|_*\le\sqrt{k} \|N\|_F \le n\sqrt{k}.
\end{equation}
Now take any $1\le i\le n$. Suppose that $i\in T_l$. Let $i' = r(l)$.
Suppose that team $i'$ is weaker than team $i$. Then $p_{i'j} \le
p_{ij}$ for all $j\ne i,i'$. Thus,
%
\begin{eqnarray}
\sum_{j=1}^n (p_{ij} -
n_{ij})^2 &=& \sum_{j=1}^n
(p_{ij} - p_{i'j})^2\le\sum
_{j=1}^n |p_{ij}- p_{i'j}|
\nonumber
\\
\label{bt2}&=&  \sum_{j\ne i,i'} (p_{ij} - p_{i'j})
+ p_{ii'}+ p_{i'i}
\\
&\le &  t_i - t_{i'} + 2 \le\frac{n}{k} + 2\le
\frac{3n}{k}.\nonumber
\end{eqnarray}
Similarly, if team $i'$ is stronger than team $i$,
%
\begin{equation}
\label{bt3} \sum_{j=1}^n
(p_{ij} - n_{ij})^2\le t_{i'}-t_i+2
\le\frac{3n}{k}.
\end{equation}
By \eqref{bt1}, \eqref{bt2}, \eqref{bt3} and \eqref{csineq0} we have
\begin{eqnarray*}
\|M\|_* &\le & \|M-N\|_* + \|N\|_*
\\
&\le & \sqrt{n} \|M-N\|_F + n\sqrt{k}
\\
&\le & \frac{3n^{3/2}}{\sqrt{k}} + n\sqrt{k}.
\end{eqnarray*}
Choosing $k= [n^{1/2}]$, we get $\|M\|_*\le Cn^{5/4}$.
Combined with \eqref{bt0}, this proves the claim.

\section*{Acknowledgments}
I would like to thank Emmanuel Cand\`es for introducing me to this topic, Andrea Montanari and Peter Bickel
for pointing out many relevant references, and Persi Diaconis for
helpful advice. Special thanks to Yaniv Plan for pointing out an
important mistake in the first draft, and to Philippe Rigollet for
correcting an error in Theorem~\ref{bradleythm}. I would also like to
thank the three anonymous referees for a long list of useful comments.




\printaddresses
\end{document}